\documentclass[12pt]{amsart}
\usepackage[dvips]{color}
\usepackage{amsmath}
\usepackage{amsxtra}
\usepackage{amscd}
\usepackage{amsthm}
\usepackage{amsfonts}
\usepackage{amssymb}
\usepackage{eucal}
\usepackage{epsfig}
\usepackage{graphics}
\usepackage{accents}

\numberwithin{equation}{section}
\newtheorem{thm}{Theorem}[section]
\newtheorem{prop}[thm]{Proposition}
\newtheorem{lem}[thm]{Lemma}

\newtheorem{cor}[thm]{Corollary}

\newcommand{\ket}[1]{{| #1 \rangle}}      %ket
\newcommand{\one}{\mathbf{1}}
\newcommand{\C}{{\mathbb C}}
\newcommand{\Z}{{\mathbb Z}}

\newcommand{\cA}{{\mathcal A}}
\newcommand{\cB}{{\mathcal B}}

\newcommand{\E}{{\mathcal E}}
\newcommand{\F}{\mathcal F}

\newcommand{\bs}{\boldsymbol}

\newcommand{\GS}{\mathfrak{S}}
\newcommand{\gl}{\mathfrak{gl}}

\newcommand{\sln}{\mathfrak{sl}}

\newcommand{\bla}{{\boldsymbol \la}}

\newcommand{\la}{\lambda}
\newcommand{\La}{\Lambda}

\newcommand{\ssE}{\textsf{E}}
\newcommand{\ssF}{\textsf{F}}

\newcommand{\ssK}{\textsf{K}}

\newcommand{\id}{{\rm id}}

\newcommand{\Sym}{\mathrm{Sym}}

\newcommand{\mc}{\mathcal}
\newcommand{\al}{\alpha}

\newcommand{\cN}{\mathcal{N}}

\newcommand{\g}{\mathfrak{g}}

\newcommand{\ev}{\mathrm{ev}}

\newcommand{\pr}{\mathrm{pr}\,}

\definecolor{jimbo}{rgb}{0,0,1}
\definecolor{jimbocomment}{rgb}{0.8,0.0,0.2}
\begin{document}

\begin{title}[Evaluation modules]
{Evaluation modules for quantum toroidal $\gl_n$ algebras}
\end{title}
\author{B. Feigin, M. Jimbo, and E. Mukhin}
\address{BF: National Research University Higher School of Economics, 
Russian Federation, International Laboratory of Representation Theory 
and \newline Mathematical Physics, Russia, Moscow,  101000,  
Myasnitskaya ul., 20 and Landau Institute for Theoretical Physics,
Russia, Chernogolovka, 142432, pr.Akademika Semenova, 1a
}
\email{bfeigin@gmail.com}
\address{MJ: Department of Mathematics,
Rikkyo University, Toshima-ku, Tokyo 171-8501, Japan}
\email{jimbomm@rikkyo.ac.jp}
\address{EM: Department of Mathematics,
Indiana University-Purdue University-Indianapolis,
402 N.Blackford St., LD 270, 
Indianapolis, IN 46202, USA}\email{emukhin@iupui.edu}

\dedicatory{To Vyjayanthi Chari on the occasion of her 60th birthday}

\begin{abstract} 
The affine evaluation map is  
a surjective homomorphism from the quantum toroidal $\gl_n$
algebra $\E'_n(q_1,q_2,q_3)$ to the quantum affine algebra $U'_q\widehat\gl_n$ at level $\kappa$
completed with respect to the homogeneous grading, where $q_2=q^2$ and 
$q_3^n=\kappa^2$.

We discuss $\E'_n(q_1,q_2,q_3)$ evaluation modules.
We give highest weights of evaluation highest weight modules.  
We also obtain the decomposition of the 
evaluation Wakimoto module with respect to a Gelfand-Zeitlin type subalgebra of 
a completion of $\E'_n(q_1,q_2,q_3)$, 
which describes a deformation of the coset theory $\widehat\gl_n/\widehat\gl_{n-1}$.
\end{abstract}
\maketitle 

\section{Introduction}
For an arbitrary complex Lie algebra $\g$ and a non-zero constant $u$, 
we have the evaluation map 
$$\g\otimes\C[t,t^{-1}]\to\g, \qquad g\otimes t^k\mapsto u^kg.$$
The evaluation map is a surjective homomorphism of Lie algebras. 
It plays a prominent role in representation theory of current algebras and various constructions in mathematics and physics.

In type A, a quantum version of the evaluation map $\overline{ev}_u:\ U_q'\widehat{\sln}_n\to U_q\gl_n$ was 
introduced in \cite{J}, see \eqref{ev1} below. 
This map is used to construct simplest possible representations of the quantum affine algebra
$U_q\widehat{\mathfrak{\sln}}_n$ 
called evaluation modules. The evaluation modules are central for many studies. 
For example, the $R$ matrix of the 
celebrated six vertex model 
is an intertwiner for tensor products of two evaluation modules. It is well known that the quantum 
evaluation map does not exist for simple Lie algebras $\g$ other than in type A.

The affine analog of the quantum evaluation map was discovered in \cite{Mi2} which we now recall.
Let $\E_n'=\E_n'(q_1,q_2,q_3)$ be the quantum toroidal algebra associated to $\gl_n$. 
It depends on complex parameters $q_1,q_2,q_3$ such that $q_1q_2q_3=1$ and a central elements $C$ 
(we set the second central element to 1), see Section \ref{sec:En}. 
Let $U'_q\widehat{\gl}_n$ be the 
quantum affine algebra associated to $\gl_n$, 
see Section \ref{sec:affine}. 
It depends on a complex parameter $q$ and it has 
a central element $C$. We always assume $q^2=q_2$. 
An easy well known fact is that there is a homomorphism of algebras 
$v:\ U'_q\widehat{\gl}_n\to \E'_n(q_1,q_2,q_3)$ 
such that $v(C)= C$, see \eqref{embed}.  

We consider $U'_q\widehat{\gl}_n$ in the Drinfeld new realization, and we denote by 
$\widetilde{U}'_q\widehat{\gl}_n$  
its completion with respect to the homogeneous grading, see \eqref{degree}. 
We also impose the following key 
relation for the central elements of $\E_n'$ and $\widetilde{U}'_q\widehat{\gl}_n$ with parameter $q_3$:
\begin{equation}\label{levels}
C^2=q_3^n.
\end{equation}
Then by \cite{Mi2} there exists a surjective algebra homomorphism depending on a non-zero complex number $u$:
$$
\ev^{(3)}_u:\ \E_n'\to \widetilde{U}'_q\widehat{\gl}_n
$$
such that $\ev^{(3)}_u\circ v=id$.  
We call the homomorphism $\ev^{(3)}_u$ the  quantum affine evaluation map.

We give formulas for the quantum affine evaluation map in Section \ref{sec:map} and 
provide the proof in Section \ref{sec:proof}. We supply a number of details which were omitted in \cite{Mi2}.

\medskip

The completed algebra $\widetilde{U}'_q\widehat{\gl}_n$ acts on highest weight $U'_q\widehat{\gl}_n$ modules and 
therefore every such representation becomes an evaluation representation of the quantum toroidal algebra with the parameters satisfying \eqref{levels}. Moreover, this representation of $\E_n'$ is also a highest weight module. 
We discuss the highest weights of evaluation modules in Section \ref{hw:sec}. 
It turns out that some evaluation representations appeared already in \cite{FJMM1}.

\medskip

The algebra $\E'_n(q_1,q_2,q_3)$ has mutually commuting subalgebras $\E'_{1,n-1}\simeq \E'_1(q_1^n,q_2,q_3q_1^{-n+1})$ 
and $\E^{0|n-1}_{n-1}\simeq \E'_{n-1}(q_1q_3^{-1/(n-1)},q_2,q_3q_3^{1/(n-1)})$. 
Under the affine evaluation map, the image of $\E'_{1,n-1}$ is a deformation of the coset W algebra $\widehat\gl_{n}/\widehat\gl_{n-1}$. 
Therefore it is important to describe the decomposition of evaluation modules with respect to these subalgebras. 

For that purpose we recall the $n$ remarkable pairwise commuting subalgebras $\E'_{1,m}(q_{1,m},q_2,q_{3,m})$ 
corresponding to the diagonal inclusions of $\gl_1$ to $\gl_n$, see \cite{FJMM}. 
We consider the Wakimoto $U'_q\widehat{\gl}_n$ modules generated from Gelfand-Zeitlin modules of $U_q{\gl}_n$ 
and the corresponding evaluation module $\widehat{\rm GZ}_{\bla^0}(u)$ of $\E'_n(q_1,q_2,q_3)$, see Section \ref{GZ:sec}. 
We describe the decomposition of 
 $\widehat{\rm GZ}_{\bla^0}(u)$ 
with respect to algebra $\otimes_{m=0}^{n-1}\E'_{1,m}(q_{1,m},q_2,q_{3,m})$, see Theorem \ref{decompose}. 

This result is  also important for applications, see \cite{FJM}. In fact, the screening operators commuting with a 
copy of quantum affine $\gl_2$ in \cite{FJM} 
are obtained from currents of $\E'_{1,2}$ acting in an $\E'_2$ evaluation Wakimoto module.

\medskip

The paper is constructed as follows. We give definitions of various quantum algebras in 
Section \ref{sec:finite} and Section \ref{sec:toroidal}. 
In Section \ref{sec:fusion} we recall the fused currents of \cite{FJMM} 
and study their commutation relations with other currents. In Section \ref{sec:A} 
we recall the definition of the subalgebras $\E'_{1,m}(q_{1,m},q_2,q_{3,m})$. 
We give the evaluation map in Section \ref{sec:map} and 
prove that it is well-defined in Section \ref{sec:proof}. 
We discuss evaluation modules in Section \ref{sec:comments}. 
We discuss evaluation highest weight modules in Section \ref{hw:sec} and   evaluation Wakimoto modules 
in Section \ref{GZ:sec}. We give a proof of the result on  evaluation Wakimoto modules in Appendix \ref{sec: Appendix}.

\section{Quantum groups}\label{sec:finite}
In this section we set up the notation for various quantum groups.

Let $n$ be a positive integer. 
For $n\ge2$, let $(a_{i,j})_{i,j\in\Z/n\Z}$ be 
the Cartan matrix of type $A^{(1)}_{n-1}$. 

Let $\C P=\oplus_{i\in \Z/n\Z} \C\varepsilon_i$ be an $n$-dimensional vector space with the chosen basis and a non-degenerate form such that $(\varepsilon_i,\varepsilon_j)=\delta_{i,j}$. Let $P=\oplus_{i\in \Z/n\Z}\Z\varepsilon_i$ be the lattice.

Set $\al_i=\varepsilon_{i-1}-\varepsilon_i$ and $\La_i=\sum_{j=0}^{i-1}\varepsilon_j$,  $1\le i \le n-1$. 
We have $(\al_i,\Lambda_j)=\delta_{i,j},$ $(\al_i,\al_j)=a_{i,j}$.

Let $ \C\bar P=(\sum_{i\in \Z/n\Z}\varepsilon_i)^\perp\subset \C P$. For $p\in\C P$ denote $\bar p\in\C \bar P$ the projection along  $\C (\sum_{i\in \Z/n\Z}\varepsilon_i)$. 
Then $\bar\al_i$, $\bar\Lambda_i$, $1\le i\le  n-1$, are simple roots and fundamental weights of $\mathfrak{sl}_n$ respectively. 

\medskip

Fix $\log q, \log d\in\C$ and set $q=e^{\log q}, d=e^{\log d}$, 
$q_1=q^{-1}d, q_2=q^2, q_3=q^{-1}d^{-1}$, 
so that $q_1q_2q_3=1$. We assume that, for rational numbers $a,b,c$, the equality
$q_1^aq_2^bq_3^c=1$ holds if and only if $a=b=c$.

\medskip

We use the standard notation $[A,B]_p=AB-pBA$ and $[r]=(q^r-q^{-r})/(q-q^{-1})$.

\subsection{Quantum algebra $U_q{\gl}_n$}
The quantum $\gl_n$ algebra $U_q{\gl}_n$ 
 has generators $e_i$, $f_i$, $q^h$, $1\le i\le n-1$, $h\in P$,
  with the defining relations
  \begin{align*}
  &q^{h}q^{h'}=q^{h+h'}, \qquad q^0=1, \qquad q^{h}e_i=q^{(h,\al_i)}e_iq^{h},
\qquad q^{h}f_i=q^{-(h, \al_i)}f_iq^{h}\,, \\
  &[e_i,f_j]=\delta_{i,j}\frac{K_i-K_i^{-1}}{q-q^{-1}}\,,\\
  &[e_i,e_j]=[f_i,f_j]=0 \qquad (|i-j|\ge 2), \\
  &[e_i,[e_i,e_j]_{q^{-1}}]_q=[f_i,[f_i,f_j]_{q^{-1}}]_q=0 \qquad (|i-j|=1),
  \end{align*} 
where $K_i=q^{\al_i}$.

The quantum $\sln_n$ algebra $U_q{\sln}_n$ is the subalgebra of  $U_q{\gl}_n$ generated by $e_i,f_i$, $K_i$, $1\le i\leq	 n-1$.
  
The element ${\sf t}=q^{\varepsilon_0+\varepsilon_1+\dots+\varepsilon_{n-1}}\in U_q{\gl}_n$ is central and split.

\subsection{Quantum affine algebra $U'_q\widehat{\gl}_n$}\label{sec:affine}
The quantum affine algebra $U'_q\widehat{\sln}_n$ in the 
Drinfeld new  realization is defined
 by generators $x^\pm_{i,k}$, $h_{i,r}$, $q^h$, $C$, where
$1\le i\le  n-1$, $k\in\Z$, $r\in\Z\setminus\{0\}$,   $h\in P$,  with the defining relations
\begin{align*}
&\text{$C$ is central},\quad
q^{h}q^{h'}=q^{h+h'},\quad
q^0=1\,, \\
%\qquad 
&q^hx^\pm_i(z)q^{-h}=q^{\pm(h,\alpha_i)}x^\pm_i(z)\, ,
\quad [q^h,h_{j,r}]=0\,,
\\
&[h_{i,r},h_{j,s}]=\delta_{r+s,0}\frac{[ra_{i,j}]}{r}
\frac{C^r-C^{-r}}{q-q^{-1}}\,,
\\
&[h_{i,r},x^{\pm}_j(z)]=\pm\frac{[ra_{i,j}]}{r}C
^{-(r\pm|r|)/2}z^rx^\pm_j(z)\,,
\\
&[x^+_i(z),x^-_j(w)]=\frac{\delta_{i,j}}{q-q^{-1}}\Bigl(
\delta\bigl(C\frac{w}{z}\bigr)\phi_i^+(w)-
\delta\bigl(C\frac{z}{w}\bigr)\phi_i^-(z)
\Bigr)\,,
\\
&(z-q^{\pm a_{ij}}w)x^\pm_i(z)x^\pm_j(w)+
(w-q^{\pm a_{ij}}z)x^\pm_j(w)x^\pm_i(z)=0\,,
\\
&[x^\pm_i(z),x^\pm_j(w)]=0\quad \text{if $a_{ij}=0$},\\
&\Sym_{z_1,z_2}[x^\pm_i(z_1),[x^\pm_i(z_2),x^\pm_j(w)]_{q}]_{q^{-1}}=0
\quad \text{if $a_{ij}=-1$}.
\end{align*}
Here we set $x^\pm_i(z)=\sum_{k\in\Z}x^\pm_{i,k}z^{-k}$, 
$\phi^{\pm}_j(z)=K_j^{\pm1}
\exp\bigl(\pm(q-q^{-1})\sum_{r>0}h_{j,\pm r}z^{\mp r}\bigr)$, where
$K_j=q^{\al_j}$, $1\le j\le n-1$.

\medskip

The quantum affine algebra $U'_q\widehat{\gl}_n$ is the algebra $U'_{q}\widehat{\sln}_n$  with additional Heisenberg generators $Z_r$, $r\in\Z\setminus\{0\}$, such that
\begin{align}\label{Zr}
[Z_k, U'_{q}\widehat{\sln}_n]=0,
\quad [Z_r,Z_s]=-\delta_{r+s,0}[nr]\frac{1}{r}\frac{C^r-C^{-r}}{q-q^{-1}}\,.
\end{align}

\medskip

The prime in the notation $U'_{q}\widehat{\sln}_n$, $U'_{q}\widehat{\gl}_n$ indicates that we do not consider the degree operator.

The element ${\sf t}=q^{\varepsilon_0+\dots+\varepsilon_{n-1}}$ is central and split in both  $U'_{q}\widehat{\sln}_n$ and $U'_q\widehat{\gl}_n$.

The subalgebra of $U'_{q}\widehat{\gl}_n$ generated by $x^{\pm}_{i,0}$, $q^h$, $1\le i\le n-1$, $h\in P$, is isomorphic to $U_q\gl_n$.

For $n>2$, the algebra $U'_q\widehat{\gl}_n$ can be described by the same relations as above by allowing $h_{0,r}$ and setting 
$\al_0=\varepsilon_{n-1}-\varepsilon_{0}$. Then we have
\begin{align*}
Z_r=\sum_{i=0}^{n-1}\frac{q^{ir}+q^{(n-i)r}}{q^r-q^{-r}}h_{i,r}.
\end{align*}

\medskip

Define the Chevalley generators of $U'_q\widehat{\sln}_n$:
\begin{align}
&e_i=x^+_{i,0},\quad f_i=x^-_{i,0} \qquad(1\le i \le n-1),
\nonumber\\
&e^{(r)}_0=
[x^-_{n-1,0}, \cdots [x^-_{2,0},x^-_{1,-1}]_{q}\cdots ]_{q}\,
q^{\al_1+\dots+\al_{n-1}}\,,\label{Chev}\\
&f^{(r)}_0=
q^{-\al_1-\dots-\al_{n-1}}
[\cdots [x^+_{1,1},x^+_{2,0}]_{q^{-1}},\cdots x^+_{n-1,0}]_{q^{-1}}
\,.
\nonumber
\end{align}
The Chevalley generators $e_j, f_j$, 
$1\le j\le n-1$, 
$e_0^{(r)},f_0^{(r)}$ generate $U'_q\widehat{\sln}_n$.

We will also use the other set of Chevalley generators which we call left Chevalley generators:
\begin{align}
&e_0^{(l)}=q^{-\al_1-\dots-\al_{n-1}}
[ \cdots [x^-_{1,1}, x^-_{2,0}]_{q}\cdots ,x^-_{n-1,0}]_{q}\,,
 \notag \\
&f_0^{(l)}=
[x^+_{n-1,0},\cdots [x^+_{2,0},x^+_{1,-1}]_{q^{-1}}\cdots ]_{q^{-1}}q^{\al_1+\dots+\al_{n-1}}.\label{left Chev}
\end{align}
The elements $e_j, f_j$, $1\le j\le n-1$,   
$e_0^{(l)},f_0^{(l)}$ generate $U'_q\widehat{\sln}_n$ as well.

\medskip
 For $u\in\C^\times$, we have the evaluation homomorphism, 
$\overline{\ev}_u:U^{'}_q\widehat{\mathfrak{sl}}_n\to U_q\mathfrak{gl}_n$, given in Chevalley generators by (see \cite{J}) 
\begin{align}
&\overline{\ev}_u(e_i)=e_i, \qquad  \overline{\ev}_u(f_i)=f_i \qquad (1\le i\le n-1), 
\nonumber %\label{ev0}
\\
&\overline{\ev}_u(e_0)=
u^{-1} q^{-\Lambda_1+\Lambda_{n-1}}
[\cdots[f_1,f_2]_{q^{-1}},\cdots,f_{n-1}]_{q^{-1}}\,,
\label{ev1}
\\
&\overline{\ev}_u(f_0)=
u \,
[e_{n-1},\cdots[e_2,e_1]_{q},\cdots]_{q}\, q^{\Lambda_1-\Lambda_{n-1}}\,.
\nonumber %\label{ev2}
\end{align} 
Note that $\overline{\ev}_u(C)=1$.
\medskip

We have  the homogeneous grading given by
\begin{align}\label{degree}
&\deg e_i=-\deg f_i=0 \qquad (1\le i\le n-1), \qquad &\deg q^h=0 \qquad (h\in P),\nonumber
\\
&\deg\, e_0^{(l)}=-\deg\, f_0^{(l)}=1, &\deg Z_r=r\qquad (r\neq0)\,.
\end{align}
We also have the principal degree given by
\begin{align}\label{pdegree}
&{\rm pdeg}\, e_i=-{\rm pdeg}\, f_i=1 \qquad (1\le i\le n-1), \qquad & {\rm pdeg}\, q^h=0 \qquad (h\in P),\notag
\\
&{\rm pdeg}\, e_0^{(l)}=-{\rm pdeg}\, f_0^{(l)}=1, & \quad {\rm pdeg}\, Z_r=nr\qquad (r\neq0)\,.
\end{align}

Then, in particular,  $\deg\, x^{\pm}_{i,k}=k$ and ${\rm pdeg}\, x^{\pm}_{i,k}=nk\pm 1$.

For $\kappa\in\C^\times$ we denote $U'_{q,\kappa}\widehat{\gl}_n$ the quotient of $U'_q\widehat{\gl}_n$ by the relation $C=\kappa$.
We denote $\widetilde{U}'_{q,\kappa}\widehat{\gl}_n$ the completion of the algebra $U'_{q,\kappa}\widehat{\gl}_n$ with respect to 
the homogeneous grading 
 in the positive direction. 
Elements of $\widetilde{U}'_{q,\kappa}\widehat{\gl}_n$ have the form $\sum_{r=r_0}^\infty g_r$, where 
$g_r\in U'_{q,\kappa}\widehat{\gl}_n$, $\deg g_r=r$.

 We denote by ${\mathfrak b}_{q,\kappa}$
the subalgebra of $U'_{q,\kappa}\widehat{\gl}_n$ generated by 
$x_{i,k}^{\pm}$,  $h_{j,r}$, $Z_r$,
$q^h$, where $1\le i\le n-1$, 
$1\le j\le n-1$, 
$k\ge 0$, 
$r\ge 1$, $h\in P$.

\section{Quantum toroidal $\gl_n$}\label{sec:En}
In this section we recall the definition of quantum toroidal algebras and some facts we will use.

\subsection{Definition of $\E'_n$}\label{sec:toroidal}
For $i,j\in\Z/n\Z$ and $r\neq 0$ we set
\begin{align*}
a_{i,j}(r)=\frac{[r]}{r} \times \left((q^r+q^{-r})\delta^{(n)}_{i,j}-d^r \delta^{(n)}_{i,j-1}-d^{-r}\delta^{(n)}_{i,j+1}\right),
\end{align*}
where $\delta^{(n)}_{i,j}$ is Kronecker's delta modulo $n$: $\delta^{(n)}_{i,j}=1$ (if $i\equiv j\bmod n$) and $\delta^{(n)}_{i,j}=0$ (otherwise).

Define further functions $g_{i,j}(z,w)$ by
\begin{align*}
n\ge 3&:\quad
g_{i,j}(z,w)=\begin{cases}
	      z-q_1w & (i\equiv j-1),\\
              z-q_2w & (i\equiv j),\\
	      z-q_3w & (i\equiv j+1),\\
              z-w & (i\not\equiv j,j\pm1),\\
	     \end{cases}\\
n=2&:\quad 
 g_{i,j}(z,w)=\begin{cases}
	      z-q_2w & (i\equiv j),\\
              (z-q_1w)(z-q_3w)& (i\not\equiv j),
	     \end{cases}\\
n=1&:\quad 
 g_{0,0}(z,w)=(z-q_1w)(z-q_2w)(z-q_3w),
\end{align*}
and set
$d_{i,j}=d^{\mp 1}$ ($i\equiv j\mp1, n\ge 3$), $d_{i,j}=-1$ ($i\not\equiv j, n=2$), and
$d_{i,j}=1$ (otherwise).

The quantum toroidal algebra  of type $\gl_n$, 
which we denote by $\E'_n=\E'_n(q_1,q_2,q_3)$, is a unital associative
algebra generated by $E_{i,k},F_{i,k},H_{i,r}$ 
and invertible elements $q^{h}$, $C$, 
where $i\in\Z/n\Z$, $k\in \Z$, $r\in\Z\backslash\{0\}$, 
$h\in P$.  
As always, we set $K_i=q^{\alpha_i}$ for $i\in\Z/ n\Z$. We have  
\begin{equation}\label{KKKK}
K_0=\prod_{i=1}^{n-1}K_i^{-1}.
\end{equation}

We present below the defining relations 
in terms of generating series
\begin{align*}
&E_i(z)=\sum_{k\in\Z}E_{i,k}z^{-k}, 
\quad
F_i(z)=\sum_{k\in\Z}F_{i,k}z^{-k},\\
&K^{\pm}_i(z)=K_i^{\pm1}\bar{K}^{\pm}_i(z)\,,
\quad
\bar{K}^{\pm}_i(z)=
\exp\Bigl(\pm(q-q^{-1})\sum_{r>0}H_{i,\pm r}z^{\mp r}\Bigr)\,.
\end{align*}
The relations are as follows.

\bigskip

\noindent{\bf $C,q^h$ relations}
\begin{align}
&\text{$C$ is central},\quad
q^{h}q^{h'}=q^{h+h'},\quad
q^0=1\,,
\label{CD1}
\\
&q^{h}E_i(z)q^{-h}=
q^{(h,\alpha_i)}E_i(z)\,,
\quad
q^{h}F_i(z)q^{-h}=
q^{-(h,\alpha_i)}F_i(z)\,,
\quad q^hK^\pm_i(z)=K^\pm_i(z)q^h\,.
\label{CD2}
\end{align}
\bigskip

\noindent{\bf $K$-$K$, $K$-$E$ and $K$-$F$ relations}
\begin{align}
&K^\pm_i(z)K^\pm_j (w) = K^\pm_j(w)K^\pm_i (z), 
\label{KK1}\\
&\frac{g_{i,j}(C^{-1}z,w)}{g_{i,j}(Cz,w)}
K^-_i(z)K^+_j (w) 
=
\frac{g_{j,i}(w,C^{-1}z)}{g_{j,i}(w,Cz)}
K^+_j(w)K^-_i (z),
\label{KK2}\\
&d_{i,j}g_{i,j}(z,w)K_i^\pm(C^{-(1\pm1)/2}z)E_j(w)+g_{j,i}(w,z)E_j(w)K_i^\pm(C^{-(1\pm1) /2}z)=0,
\label{KE}\\
&d_{j,i}g_{j,i}(w,z)K_i^\pm(C^{-(1\mp1)/2}z)F_j(w)+g_{i,j}(z,w)F_j(w)K_i^\pm(C^{-(1\mp1) /2}z)=0\,.
\label{KF}
\end{align}
\bigskip

\noindent{\bf $E$-$F$ relations}
\begin{align}
&[E_i(z),F_j(w)]=\frac{\delta_{i,j}}{q-q^{-1}}
(\delta\bigl(C\frac{w}{z}\bigr)K_i^+(w)
-\delta\bigl(C\frac{z}{w}\bigr)K_i^-(z))\,.
\label{EF}
\end{align}
\bigskip

\noindent{\bf $E$-$E$ and $F$-$F$ relations}
\begin{align*}
&[E_i(z),E_j(w)]=0\,, \quad [F_i(z),F_j(w)]=0\, \quad (i\not\equiv j,j\pm1)\,,
\\
&d_{i,j}g_{i,j}(z,w)E_i(z)E_j(w)+g_{j,i}(w,z)E_j(w)E_i(z)=0, 
\\
&d_{j,i}g_{j,i}(w,z)F_i(z)F_j(w)+g_{i,j}(z,w)F_j(w)F_i(z)=0.
\end{align*}
\bigskip

\noindent{\bf Serre relations}
For $n\ge3$,
\begin{align*}
&\Sym_{{z_1,z_2}}[E_i(z_1),[E_i(z_2),E_{i\pm1}(w)]_q]_{q^{-1}}=0\,,
\\
&\Sym_{{z_1,z_2}}[F_i(z_1),[F_i(z_2),F_{i\pm1}(w)]_q]_{q^{-1}}=0\,.
\end{align*}

For $n=2$, $i\not\equiv j$,
\begin{align*}
&\Sym_{z_1,z_2,z_3}[E_i(z_1),[E_i(z_2),[E_i(z_3),E_j(w)]_{q^2}]]_{q^{-2}} =0\,,
\\
&\Sym_{z_1,z_2,z_3}[F_i(z_1),[F_i(z_2),[F_i(z_3),F_j(w)]_{q^2}]]_{q^{-2}} =0\,.
\end{align*}

For $n=1$, 
\begin{align*}
&\Sym_{z_1,z_2,z_3}z_2z_3^{-1}[E_0(z_1),[E_0(z_2),E_0(z_3)]] =0\,,
\\
&\Sym_{z_1,z_2,z_3}z_2z_3^{-1}[F_0(z_1),[F_0(z_2),F_0(z_3)]] =0\,.
\end{align*}
In the above we set
\begin{align*}
&\Sym\ f(x_1,\dots,x_N) =\frac{1}{N!}
\sum_{\pi\in\GS_N} f(x_{\pi(1)},\dots,x_{\pi{(N)}})\,.
\end{align*}
\bigskip

Under the $C,q^h$ relations \eqref{CD2}, 
the $K$-$K$, $K$-$E$ and $K$-$F$ relations 
\eqref{KK1}--\eqref{KF} are equivalently written as
\bigskip

\noindent{\bf $H$-$E$, $H$-$F$, and $H$-$H$ relations}\quad
For $r\neq 0$, 
\begin{align*}
&[H_{i,r},E_j(z)]= a_{i,j}(r)C^{-(r+|r|)/2}  
\,z^r E_j(z)\,,
\\
&[H_{i,r},F_j(z)]=-a_{i,j}(r)C^{-(r-|r|)/2}   
\,z^r F_j(z)\,,
\\
&[H_{i,r},H_{j,s}]=\delta_{r+s,0} \cdot a_{i,j}(r)\eta_r\,,
\quad \eta_r=\frac{C^r-C^{-r}}{q-q^{-1}}\,.
\end{align*}

The $E$-$E$ and $F$-$F$ relations with $j\equiv i\pm 1$ in Fourier components read for $n\ge3$ 
\begin{align}\label{quad-rel}
&[E_{i,k+1},E_{i+1,l}]_{q^{-1}}=
q_1 [E_{i,k},E_{i+1,l+1}]_{q}\,,
\\
&[F_{i,k+1},F_{i+1,l}]_{q}=
q_3^{-1} [F_{i,k},F_{i+1,l+1}]_{q^{-1}}\,,\notag
\end{align} 
and for $n=2$
\begin{align}\label{quad-rel2}
&[E_{i,k+1},E_{i+1,l-1}]_{q^{-2}}
-(q_1+q_3)[E_{i,k},E_{i+1,l}]
+q_1q_3[E_{i,k-1},E_{i+1,l+1}]_{q^{2}}=0\,,
\\
&[F_{i+1,l-1},F_{i,k+1}]_{q^{-2}}
-(q_1+q_3)[F_{i+1,l},F_{i,k}]
+q_1q_3[F_{i+1,l+1},F_{i,k-1}]_{q^{2}}=0\,.\notag
\end{align}

\bigskip

The algebra $\E'_n$ considered here is obtained from 
\cite{FJMM} by setting the second central element $\prod_{i=0}^{n-1}K_i$ to $1$, 
dropping the scaling elements
$D$, $D^\perp$ and adding the split central element ${\sf t}=q^{{\varepsilon}_0+\dots+\varepsilon_{n-1}}$. 
Our generators $E_i(z),F_i(z),K_i^\pm(z), C$ 
correspond to the perpendicular generators $E^\perp_i(z),F^\perp_i(z),K^{\pm,\perp}_i((q^c)^\perp z), (q^c)^\perp$
of \cite{FJMM}.

\medskip

Algebra $\E_n'$ is $\Z$-graded by  
\begin{align}\label{tor degree}
\deg E_{i,k}=\deg F_{i,k}=k, \qquad \deg H_{i,r}=r, \qquad  \deg C=\deg q^h=0. 
\end{align}
We denote $\widetilde{\E}_n'$ the completion of $\E_n'$ with respect to this grading in the positive direction.
   
We have a graded embedding $v:\ U'_q\widehat{\gl}_n\to \E'_n$ given by
\begin{align}\label{embed}
&x_i^+(z)\mapsto E_i(d^{-i}z)\,,\quad  
x_i^-(z)\mapsto F_i(d^{-i}z)\,,\quad  
\phi_i^\pm(z)\mapsto K^\pm_i(d^{-i}z)\,,
\end{align}
and $C\mapsto C$,  
$q^h\mapsto q^h$. We call the image of $v$ the vertical subalgebra
and denote it by $U^v_q\widehat{\gl}_n$. 

We call the subalgebra of $\E'_n$ generated by $E_{i,0},F_{i,0}$, $0\le i\le n-1$,
the horizontal subalgebra. 
The horizontal subalgebra, which we denote by  $U^h_q\widehat{\sln}_n$, 
is isomorphic to the quotient of 
$U_q'\widehat{\sln}_n$ 
by the relation $C=1$. 

There exists an isomorphism of algebras 
\begin{align*}
\iota:\E'_n(q_1,q_2,q_3)\to \E'_n(q_3,q_2,q_1),
\end{align*}
given by 
\begin{align}
\iota:E_i(z)\mapsto E_{n-i}(z), \quad
F_i(z)\mapsto F_{n-i}(z),\quad  
K_i^{\pm}(z)\mapsto K_{n-i}^{\pm}(z)\,,
\label{iota}
\end{align}
and  $\iota(C)=C$. 

We have also the Miki automorphism $\theta$ 
which interchanges vertical and horizontal subalgebras, see \cite{Mi}. 
We fix the definition  of $\theta$ as in \cite{FJMM}.
We remark that $\theta$ is defined for algebra $\E_n$, not $\E'_n$, but 
expressions such as $\theta^{-1}(K_i^\pm(z))$ have a well-defined meaning in $\E'_n$ as well.
In particular, we have for $n\ge 2$,
\begin{align}
&\theta^{-1}(H_{i,1})=
-(-d)^{-i}[[\cdots[[\cdots [F_{0,0},F_{n-1,0}]_{q},\cdots,F_{i+1,0}]_{q},
F_{1,0}]_{q},\cdots,F_{i-1,0}]_{q},F_{i,0}]_{q^{2}}\,,
\label{Hi1}\\
&\theta^{-1}(H_{i,-1})=
-(-d)^{i}[E_{i,0},[E_{i-1,0},\cdots,[E_{1,0},[E_{i+1,0},\cdots,
[E_{n-1,0},E_{0,0}]_{q^{-1}},\cdots]_{q^{-1}}]_{q^{-1}},\cdots]_{q^{-1}}]_{q^{-2}}\,
\label{Hi-1}
\end{align}
for $1\le i\le n-1$, and
\begin{align}
&\theta^{-1}(H_{0,1})=
-(-d)^{-n+1}[[\cdots
[F_{1,1},F_{2,0}]_{q},\cdots,F_{n-1,0}]_{q},F_{0,-1}]_{q^2}\,,
\label{H01}
\\
&\theta^{-1}(H_{0,-1})=
-(-d)^{n-1}[E_{0,1},[E_{n-1,0},\cdots,[E_{2,0},E_{1,-1}]_{q^{-1}}
\cdots ]_{q^{-1}}]_{q^{-2}}\,.
\label{H0-1}
\end{align}
In addition 
\begin{align}
&\theta(E_{0,0})=d^{-1}[F_{n-1,0},\ldots,[F^-_{2,0},F^-_{1,-1}]_q\ldots]_qK_1\cdots K_{n-1}\,,
\label{thetaE}\\ 
&\theta(F_{0,0})=d(K_1\cdots K_{n-1})^{-1}[\ldots [E_{1,1},E_{2,0}]_{q^{-1}},\ldots,E_{n-1,0}]_{q^{-1}}
\label{thetaF}\,.
\end{align}

We define  $U^h_q\widehat{\gl}_n$ to be the subalgebra generated by
 $U^h_q\widehat{\sln}_n$ and $\theta^{-1}(K_0^\pm(z))$.  % $0\le i\le n-1$.

\subsection{Fused currents}\label{sec:fusion}
We recall the fused currents of \cite{FJMM} and compute the commutation relations with generators of 
$U'_q\widehat{\gl}_n$. 

It is convenient to use series of generators of $U'_q\widehat{\gl}_n$ defined by
$E_i(z)=x_i^+(d^iz)$, $F_i(z)=x_i^-(d^iz)$, 
$K^\pm_i(z)=\phi_i^\pm(d^iz)$, $1\le i\le n-1$, cf. \eqref{embed}.

Following \cite{FJMM} let us introduce the following elements
of  $\widetilde{U}'_q\widehat{\gl}_n$: 
\begin{align}
&\ssE(z)=
\prod_{i=1}^{n-2}\bigl(1-\frac{z_{i}}{z_{i+1}}\bigr)\cdot
E_{n-1}(q_3^{n/2-1}z_{n-1})\cdots 
E_2(q_3^{-n/2+2}z_2)E_1(q_3^{-n/2+1}z_1)
\Bigl|_{z_1=\cdots=z_{n-1}=z}\,,
\label{ssE}\\
&\ssF(z)=
\prod_{i=1}^{n-2}\bigl(1-\frac{z_{i+1}}{z_{i}}\bigr)\cdot
F_1(q_3^{-n/2+1}z_1)F_2(q_3^{-n/2+2}z_2)
\cdots  F_{n-1}(q_3^{n/2-1}z_{n-1})
\Bigl|_{z_1=\cdots=z_{n-1}=z}\,,
\label{ssF}\\
&\ssK^\pm(z)=\prod_{i=1}^{n-1}K^\pm_i\bigl(q_3^{-n/2+i}z\bigr)\,.
\label{ssK}
\end{align}
When $n=2$ we have $\ssE(z)=E_1(z)$, $\ssF(z)=F_1(z)$,
$\ssK^\pm(z)=K^\pm_1(z)$.

The following result is a special case of the construction in 
\cite{FJMM}.
\begin{prop}\cite{FJMM}
For $n\ge2$, the currents \eqref{ssE}--\eqref{ssK} satisfy
\begin{align}
&[\ssE(z),\ssF(w)]=\frac{1}{q-q^{-1}}\Bigl(
\delta\bigl(C\frac{w}{z}\bigr)\ssK^+(z)
-
\delta\bigl(C\frac{z}{w}\bigr)\ssK^-(z)
\Bigr) \,,
\label{EFfuse}\\
&(z-q_2w)\ssE(z)\ssE(w)+(w-q_2z)\ssE(w)\ssE(z)=0\,,
\label{EEfuse}\\
&(w-q_2z)\ssF(z)\ssF(w)+(z-q_2w)\ssF(w)\ssF(z)=0\,,
\label{FFfuse}\\
&[\ssE(z),E_i(w)]=[\ssE(z),F_i(w)]=0\,
\quad (2\le i\le n-2),
\label{EEiFi}\\
&[E_i(z),\ssF(w)]=[F_i(z),\ssF(w)]=0\,\quad (2\le i\le n-2).
\label{FEiFi}
\end{align} 
\end{prop}

In addition we calculate the other commutation relations.
\begin{prop}
For $n\ge3$, we have
\begin{align}
&(z-q_3^{-n/2}q_1^{-1}w)E_1(z)\ssE(w)=q(z-q_3^{-n/2+1}w)\ssE(w)E_1(z)\,,
\label{E1E}\\
&(z-q_3^{n/2-1}w)E_{n-1}(z)\ssE(w)=q(z-q_3^{n/2}q_1^{-1}w)\ssE(w)E_{n-1}(z)\,,
\label{EnE}\\
&(z-q_3^{-n/2+1}w)F_1(z)\ssF(w)=q^{-1}(z-q_3^{-n/2}q_1^{-1}w)\ssF(w)F_1(z)\,,
\label{F1F}\\
&(z-q_3^{n/2}q_1w)F_{n-1}(z)\ssF(w)=q^{-1}(z-q_3^{n/2-1}w)\ssF(w)F_{n-1}(z)\,,
\label{FnF}\\
&(z-C^{-1}q_3^{n/2-1}w)[\ssE(z),F_1(w)]=
(z-Cq_3^{-n/2+1}w)[\ssE(z),F_{n-1}(w)]=0\,,
\label{EF1n}\\
&(z-Cq_3^{-n/2+1}w)[E_1(z),\ssF(w)]=
(z-C^{-1}q_3^{n/2-1}w)[E_{n-1}(z),\ssF(w)]=0\,.
\label{FE1n}
\end{align}
\end{prop}
\begin{proof}
As an example we consider \eqref{EF1n}. 
From the defining relations \eqref{EF} we have
\begin{align*}
&
(q-q^{-1})
[E_{n-1}(q_3^{n/2-1}z_{n-1})\cdots E_1(q_3^{-n/2+1}z_1),F_1(w)]
\\
&=
E_{n-1}(q_3^{n/2-1}z_{n-1})\cdots E_{2}(q_3^{-n/2+2}z_{2})
\Bigl(\delta\bigl(Cq_3^{n/2-1}\frac{w}{z_1}\bigr) K^+_1(w)
-
\delta\bigl(Cq_3^{-n/2+1}\frac{z_1}{w}\bigr) K^-_1(q_3^{-n/2+1}z_1)
\Bigr)\,.
\end{align*}
Upon multiplying by $z_1-C^{-1}q_3^{n/2-1}w$,  the second term in the right hand side 
vanishes. 
The first term does not have a pole at $z_1=z_2$
because $K_1^+(w)$ is a power series in $w^{-1}$ placed at the rightmost.  
Multiplying further by $\prod_{i=1}^{n-2}(1-z_i/z_{i+1})$ and setting 
$z_1=\cdots=z_{n-1}=z$ we find 
\[
(z-C^{-1}q_3^{n/2-1}w)[\ssE(z),F_1(w)]=0
\]
as desired.

The rest of the relations can be shown in a similar manner. 
\end{proof}

\subsection{The subalgebra $\mathcal A$}\label{sec:A}
We also recall from \cite{FJMM} the fusion construction of subalgebras 
\begin{align*}
&\E'_{1,m}=\E'_1(q_{1,m},q_2,q_{3,m}) ~
\subset \widetilde{\E}_n'(q_1,q_2,q_3)
\quad (0\le m\le n-1),
\\
&q_{1,m}=q_1^{m+1}q_3^{-n+m+1}\,,\quad 
q_{3,m}=q_3^{n-m}q_1^{-m}\,.
\end{align*}
which mutually commute and intersect only by the central element $C$.

The algebra $\E'_{1,m}$ is generated by the fused currents 
\begin{align*}
-\frac{qq_{3,m}}{1-q_{3,m}}
\ssE_m(z)&=\prod_{i=0}^{n-2}\bigl(1-\frac{z_{i}}{z_{i+1}}\bigr)\cdot
 E_{0}(q_3^{n}z_0) E_{n-1}(q_3^{n-1}z_{n-1}) \cdots E_{m+1}(q_3^{m+1}z_{m+1})
\\
&\times 
E_1(q_3^{n}q_1^{-1}z_1) \cdots E_{m}(q_3^{n}q_1^{-m}z_{m}) 
\Bigl|_{z_0=\cdots=z_{n-1}=z}\,,
\\
(1-q_{1.m})
\ssF_m(z)&=\prod_{i=0}^{n-2}\bigl(1-\frac{z_{i+1}}{z_{i}}\bigr)\cdot
F_{m}(q_3^{n}q_1^{-m}z_{m})\cdots F_1(q_3^{n}q_1^{-1}z_1)
\\
& \times 
F_{m+1}(q_3^{m+1}z_{m+1})\cdots F_{n-1}(q_3^{n-1}z_{n-1})F_{0}(q_3^{n}z_0)
\Bigl|_{z_0=\cdots=z_{n-1}=z}\,,
\\
\ssK_m^\pm(z)&=\prod_{i=m+1}^{n-1}K_i^\pm(q_3^iz)\ \prod_{i=0}^mK_i^\pm(q_1^{-i}q_3^nz).
\end{align*}

Let 
$\mathcal{A}$ be the algebra generated by subalgebras $\E'_{1,m}$,  $0\le m\le n-1$.

We have 
\begin{equation}\label{A}
\mathcal{A}=  \E'_{1,0}\otimes\E'_{1,1}\otimes\cdots\otimes \E'_{1,n-1}\subset \widetilde{\E}'_n(q_1,q_2,q_3)\,.
\end{equation}
The algebra $\mc A$ can be considered as an analog of the
Gelfand-Zeitlin subalgebra in $\widetilde{\E}'_n$.

\section{Quantum affine evaluation map}\label{sec:map}
In this section we define and prove the evaluation map in the quantum toroidal setting.

\subsection{The definition of the quantum affine evaluation map}
From now on, we consider the quotient algebra of $\E'_n$ by the relation
\begin{align*}
C=q_3^{n/2}\,.
\end{align*}
Denote this quotient by $\E^{(3)}_{n}$. 

Introduce currents $A_\pm(z)=\sum_{r>0}A_{\pm r}z^{\mp r}$,
$B_\pm(z)=\sum_{r>0}B_{\pm r}z^{\mp r}$ in 
$U'_q\widehat{\gl}_n$ by setting 
\begin{align*}
&A_{-r}=\eta_r^{-1} C^{-r}
(H_{0,-r}+\sum_{i=1}^{n-1}q_3^{ir}H_{i,-r})\,,
\quad
A_{r}=-\eta_r^{-1} (H_{0,r}+\sum_{i=1}^{n-1}q_3^{(n-i)r}H_{i,r})\,,
\\
&B_{-r}=-\eta_r^{-1} (H_{0,-r}+\sum_{i=1}^{n-1}q_3^{-(n-i)r}H_{i,-r})\,,
\quad
B_r=\eta_r^{-1}  C^{r}
(H_{0,r}+\sum_{i=1}^{n-1}q_3^{-ir}H_{i,r})\,.
\end{align*}
Set further 
\begin{align*}
&\mathcal{K}
=q^{-{\Lambda}_1+{\Lambda}_{n-1}}.
\end{align*}
Note that $\mathcal{K}$ commutes with $\ssE(z),\ssF(z)$.

\begin{thm}\cite{Mi2} \label{thm:eval}
Let $u\in\C^\times$, and set $\kappa=q^{n/2}_3$. 
The following assignment gives a homomorphism of algebras 
$\ev^{(3)}_u:\E^{(3)}_n\to \widetilde{U}'_{q,\kappa}\widehat{\gl}_n$ such that 
$\ev^{(3)}_u\circ v=\id$:
\begin{align}
&E_0(z)\mapsto u^{-1}\, e^{A_-(z)}\ssF(z)e^{A_+(z)}
\mathcal{K}
\,,
\quad 
F_0(z)\mapsto u\  e^{B_-(z)}\ssE(z)e^{B_+(z)}
\mathcal{K}^{-1}
\,,
\label{E0F0}\\
&E_i(z)\mapsto E_i(z)\,,\quad F_i(z)\mapsto F_i(z)\quad (i=1,\dots,n-1),
\\
&K^\pm_i(z)\mapsto K_i^\pm(z)\quad (i=0,1,\dots,n-1),\quad
q^h\mapsto q^h\,\quad (h\in P)\,.
\end{align}
\end{thm}

\medskip

\noindent{\it Remark.}\quad 
When $n=1$ we set formally $\ssE(z)=\ssF(z)=\mathcal{K}=1$. Then 
\eqref{E0F0} is nothing but the known vertex operator realization
 of $\E'_1$ 
for $C=q_3^{1/2}$. 

\medskip 

A proof of Theorem \ref{thm:eval} is provided in Section \ref{sec:proof}.

\medskip 

In the above theorem we have chosen the currents $E_0(z), F_0(z)$
to play a special role. In view of the cyclic symmetry of $\E_n'$ which sends 
$E_i(z)\mapsto E_{i+1}(z)$,
$F_i(z)\mapsto F_{i+1}(z)$, 
$K^\pm_i(z)\mapsto K^\pm_{i+1}(z)$, 
we could have started with $E_i(z), F_i(z)$ for any $i$. 

Using the isomorphism $\iota$ \eqref{iota} which interchanges $q_1$ with $q_3$,  
it is easy to write another evaluation homomorphism $\ev^{(1)}_u$
when $$C=q_1^{n/2}.$$ 
This is parallel to the fact that there are two evaluation homomorphisms $U'_q\widehat{\sln}_n\to U_q\gl_n$.

We also remark that the evaluation map is clearly graded with respect to the homogeneous degree, see \eqref{degree}, \eqref{tor degree}, and 
commutes with the automorphism which rescales the
spectral parameter, see (2.7) in \cite{FJMM}.

\subsection{Proof}\label{sec:proof}
In this section, we prove Theorem \ref{thm:eval}.
To simplify the notation we consider $\ev=\ev^{(3)}_1$. 

We shall need commutation relations between 
$A_\pm(z),B_{\pm}(z)$ and $E_i(w),F_i(w)$. 

First we have:
$$ [A_\pm(z),E_i(w)]=[A_\pm(z),F_i(w)]=[B_\pm(z),E_i(w)]=[B_\pm(z),F_i(w)]=0 \qquad (2\le i\le n-2).$$
 
Other relations are given as follows.
\begin{align}
&e^{A_+(z)}E_1(w)e^{-A_+(z)}=\frac{z-q_3^{-1}w}{z-q_1w}E_1(w)\,,
\quad 
e^{A_+(z)}F_1(w)e^{-A_+(z)}
=\frac{z-Cq_1w}{z-Cq_3^{-1}w}F_1(w)\,,
\label{A+EF}
\\
&e^{B_+(z)}E_{n-1}(w)e^{-B_+(z)}=
\frac{z-C^{-1}q_1^{-1}w}{z-C^{-1}q_3w}E_{n-1}(w)\,,
\quad
e^{B_+(z)}F_{n-1}(w)e^{-B_+(z)}=\frac{z-q_3w}{z-q_1^{-1}w}F_{n-1}(w)\,,
\label{B+EF}
\\
&e^{-A_-(z)}E_{n-1}(w)e^{A_-(z)}=
\frac{w-q_3^{-1}z}{w-q_1z}E_{n-1}(w)\,,
\quad
e^{-A_-(z)}F_{n-1}(w)e^{A_-(z)}=
\frac{w-Cq_1z}{w-Cq_3^{-1}z}F_{n-1}(w)\,,
\label{A-EF}\\
&e^{-B_-(z)}E_1(w)e^{B_-(z)}
=\frac{w-C^{-1}q_1^{-1}z}{w-C^{-1}q_3z}E_1(w)\,,
\quad 
e^{-B_-(z)}F_1(w)e^{B_-(z)}
=\frac{w-q_3z}{w-q_1^{-1}z}F_1(w)\,.
\label{B-EF}
\end{align}
We have also
\begin{align}
&e^{A_+(z)}e^{A_-(w)}=\frac{(z-w)^2}{(z-q_2w)(z-q_2^{-1}w)}
e^{A_-(w)}e^{A_+(z)}\,,
\label{A+A-}\\
&e^{A_+(z)}e^{B_-(w)}
\frac{(z-Cq_2w)(z-C^{-1}q_2^{-1}w)}{(z-Cw)(z-C^{-1}w)}
e^{B_-(w)}e^{A_+(z)}\,,
\label{A+B-}\\
&e^{B_+(z)}e^{A_-(w)}=
\frac{(z-Cq_2w)(z-C^{-1}q_2^{-1}w)}{(z-Cw)(z-C^{-1}w)}
e^{A_-(w)}e^{B_+(z)}
\,,
\label{B+A-}\\
&e^{B_+(z)}e^{B_-(w)}=\frac{(z-w)^2}{(z-q_2w)(z-q_2^{-1}w)}
e^{B_-(w)}e^{B_+(z)}\,.
\label{B+B-}
\end{align}

Let us verify the relations involving $E_0(z),F_0(z)$ case-by-case.
\bigskip

\noindent{\it  $C,q^h$ relations.}\quad
These are easy to check. 
\bigskip

\noindent{\it $H$-$E$ and $H$-$F$ relations.}\quad
The relation
\begin{align*}
[\ev(H_{i,r}),\ev(E_0(z))]=a_{i,0}(r)C^{-r}z^r \ev(E_0(z)) \quad (r\ge 1)
\end{align*}
follows from $C=q_3^{n/2}$  and
\begin{align*}
&[H_{i,r},e^{A_-(z)}]=z^re^{A_-(z)}\times
C^{-r}\bigl(a_{i,0}(r)+\sum_{j=1}^{n-1}a_{i,j}(r)
q_3^{jr}\bigr)\,,\\
&[H_{i,r},\ssF(z)]=-z^r\ssF(z)\times\sum_{j=1}^{n-1}
 a_{i,j}(r)\bigl(q_3^{-n/2+j}\bigr)^r\,.
\end{align*}
The relations for $r\le -1$ and for 
$[\ev(H_{i,r}),\ev(F_0(z))]$ can be verified in a similar manner.
\bigskip

\noindent{\it  $E$-$F$ relations.}\quad
First consider the relation
\begin{align*}
[\ev(E_0(z)),\ev(F_i(w))] =0\quad (i\neq 0).
\end{align*}
We have
\begin{align*}
&\ev(E_0(z))\ev(F_i(w))=e^{A_-(z)}\ssF(z)F_i(w)e^{A_+(z)}\mathcal{K}
\times \begin{cases}
	q(z-Cq_1w)/(z-Cq_3^{-1}w)& (i=1)\\
        q^{-\delta_{i,n-1}}& (2\le i\le n-1)\\
       \end{cases}\,,
\\
&\ev(F_i(w))\ev(E_0(z))=e^{A_-(z)}F_i(w)\ssF(z)e^{A_+(z)}\mathcal{K}
\times \begin{cases}
	1& (1\le i\le n-2)\\
        (w-Cq_1z)/(w-Cq_3^{-1}z)& (i= n-1)\\
       \end{cases}\,,
\end{align*}
so that the relations reduce to \eqref{F1F}, \eqref{FnF} and \eqref{FEiFi}. 
The case $[\ev(E_i(z)),\ev(F_0(w))] =0$ ($i\neq0$) is similar.

Using \eqref{A+EF}, \eqref{B-EF} and \eqref{A+B-} we obtain
\begin{align*}
\ev(E_0(z))\ev(F_0(w))&=
e^{A_-(z)}\ssF(z)e^{A_+(z)}\mathcal{K}\cdot
e^{B_-(z)}\ssE(z)e^{B_+(z)}\mathcal{K}^{-1}
\\
&=e^{A_-(z)+B_-(w)}\ssF(z)\ssE(w)e^{A_+(z)+B_+(w)}\,.
\end{align*}
Computing $\ev(F_0(w))\ev(E_0(z))$
similarly and using further \eqref{EFfuse}, we find
\begin{align*}
&[\ev(E_0(z)),\ev(F_0(w))]= \frac{1}{q-q^{-1}}\\
&\times
e^{A_-(z)+B_-(w)}\bigl(-\delta\Bigl(C\frac{z}{w}\bigr)\ssK^+(z)
+\delta\bigl(C\frac{w}{z}\bigr)\ssK^-(w)
\Bigr)
e^{A_+(z)+B_+(w)}\,.
\end{align*}
Noting that
\begin{align*}
&e^{A_-(z)+B_-(Cz)}=\bar K_0^-(z)\,,\quad
 e^{A_+(Cw)+B_+(w)}=\bar K_0^+(w)\,,
\\
&e^{A_-(Cw)+B_-(w)}=K_0 \ssK^-(w)^{-1}\,,
\quad
 e^{A_+(z)+B_+(Cz)}=K_0^{-1}\ssK^+(z)^{-1}\,,
\end{align*}
we obtain the desired result.
\bigskip

\noindent{\it $E$-$E$ and $F$-$F$ relations.}\quad
To check the quadratic relations
\begin{align*}
d_{0,j}g_{0,j}(z,w)\ev(E_0(z))\ev(E_j(w))+
g_{j,0}(w,z)\ev(E_j(w))\ev(E_0(z))=0,  
\end{align*}
we proceed in the same way as above; using \eqref{A+EF}, \eqref{A-EF}
we bring $A_+(z)$ to the right, $A_-(z)$ to the left, and apply
\eqref{E1E}, \eqref{EnE}. Verification of the
$F$-$F$ relations is entirely similar.
\bigskip

\noindent{\it Serre relations.}\quad Let us check the Serre relations 
assuming $n\ge3$. We have
\begin{align*}
&e^{-A_-(w)}  
[
\ev(E_1(z_1)),[\ev(E_1(z_2)),\ev(E_0(w))]_q]_{q^{-1}}\
e^{-A_+(w)}\\
&=  
E_1(z_1)E_1(z_2)\ssF(w)-(q+q^{-1})q^{-1}\frac{w-q_3^{-1}z_2}{w-q_1z_2}
E_1(z_1)\ssF(w)E_1(z_2)\\
&+q^{-2}\frac{w-q_3^{-1}z_1}{w-q_1z_1}\frac{w-q_3^{-1}z_2}{w-q_1z_2}
\ssF(w)E_1(z_1)E_1(z_2)\,.
\end{align*}
In view of \eqref{FE1n}, we can move $\ssF(w)$ to the right without producing
delta functions. After simplification, the right hand side becomes
\begin{align*}
-\frac{(1-q_2^{-1})w}{(w-q_1z_1)(w-q_1z_2)}(z_1-q_2z_2)E_1(z_1)E_1(z_2)\ssF(w)\,. 
\end{align*}
Symmetrizing in $z_1,z_2$ we obtain $0$ due to the quadratic relations for
$E_1(z)$. 

Likewise we compute
\begin{align*}
&e^{-A_-(z_1)-A_-(z_2)} [\ev(E_0(z_1)),[\ev(E_0(z_2)),
\ev(E_1(w))]_q]_{q^{-1}}\
e^{-A_+(z_1)-A_+(z_2)}\\
&=\frac{z_1-q_2^{-1}z_2}{z_1-q_2z_2}
\Bigl(
q^{-2}\frac{z_1-q_3^{-1}w}{z_1-q_1w}\frac{z_2-q_3^{-1}w}{z_2-q_1w}
\ssF(z_1)\ssF(z_2)E_1(w)\\
&-(q+q^{-1})q^{-1}\frac{z_1-q_3^{-1}w}{z_1-q_1w}
\ssF(z_1)E_1(w)\ssF(z_2)
+E_1(w)\ssF(z_1)\ssF(z_2)\Bigr)\,
\\
&=
\frac{q_2^{-1}(q_1-q_3^{-1})w}
{(z_1-q_1w)(z_2-q_1w)}(z_1-q_2^{-1}z_2)E_1(w)\ssF(z_1)\ssF(z_2)\,. 
\end{align*}
Due to \eqref{FFfuse}, the last line vanishes 
after symmetrization. 

Serre relations in the remaining cases (including the case $n=2$)
can be verified by the same argument.
We omit further details.
\bigskip

The proof is over.

\section{Evaluation modules}\label{sec:comments}
In this section we define and discuss the evaluation modules.

\subsection{Evaluation modules}
Recall the grading of $U'_q\widehat{\gl}_n$  given by \eqref{degree}.
We say that a $U'_q\widehat{\gl}_n$ module $V$ is  admissible if 
for any $v\in V$ there exists an $N$ such that $xv=0$ holds for any
$x\in U'_q\widehat{\gl}_n$ with $\deg x>N$. 
Algebra $\widetilde{U}'_{q,\kappa}\widehat{\gl}_n$ has a well-defined action
on admissible modules of level $\kappa$.

The quantum affine evaluation map 
$\ev^{(3)}_u$ goes from the quantum toroidal 
$\gl_n$ algebra $\E'_n(q_1,q_2,q_3)$ 
 to the (completed) quantum affine $\gl_n$ algebra, provided that $q_3$ has the special value related to the central charge by 
$C=q_3^{n/2}$. 
Note that for the quantum affine algebra 
the value for the central element is completely arbitrary as $q$ and $q_3$ are independent variables.

It follows that any admissible representation $V$ of $U_q'\widehat{\gl}_n$ 
on which $C$ acts as an arbitrary scalar $\kappa$
can be pulled back by $\ev^{(3)}_u$ to a representation of $\E^{(3)}_n$, 
by choosing $q_3$ so that $\kappa=q_3^{n/2}$. 
We call the resulting $\E^{(3)}_n$ module the evaluation module and denote it by $V(u)$.

\subsection{Highest weight evaluation modules}\label{hw:sec}
An example of admissible modules is given by highest weight modules.

A  $U_q'\widehat{\gl}_n$ module $V$ is called a highest weight module of highest weight
$(\kappa_0,\ldots,\kappa_{n-1})\in(\C^\times)^n$ with highest weight vector $v$
if $v$ is a cyclic vector in $V$ satisfying
\begin{align*}
&e_j v=0\quad (1\le j\le n-1)\,, \qquad e_0^{(l)}v=0\,,
\qquad Z_rv=0\quad (r\ge1),
\\
&K_iv= \kappa_iv \quad (1\le i \le n-1),\qquad 
Cv=\prod_{i=0}^{n-1}\kappa_i\cdot v\,,
\end{align*}
where $Z_r$ is defined in \eqref{Zr} and $e_0^{(l)}$ is the left Chevalley generator, \eqref{left Chev}.

Highest weight $U_q'\widehat{\gl}_n$ modules are principally graded, see \eqref{pdegree}.
\medskip

An $\E'_n$ module $V$ is called a highest weight module with highest weight vector $v$
if $v$ is a cyclic vector in $V$ satisfying
 \begin{align*}
\theta^{-1}(E_i(z))v=0, \qquad
\theta^{-1}(K^{\pm}_i(z))v=P_i^\pm(z)v\qquad (0\le i \le n-1)\,.  
\end{align*}
Here $P^\pm_{i}\in\C[[z^{\mp 1}]]$ and $\theta$ is the Miki automorphism, see Section \ref{sec:fusion}.

We set the degree of the highest weight vectors to zero. Then all
highest weight weight $\E'_n$ modules are graded, see \eqref{tor degree}: 
$V=\oplus_{k\le 0}V_k$.

Let $V$ be an irreducible highest weight $\E'_n$ module such that $\dim V_k<\infty$ for all $k$. Then in the
terminology of \cite{FJMM1} the module $V$ is quasi-finite. Moreover, by Theorem 2.3 in \cite{FJMM1}, the series $P_i^\pm(z)$ are expansions of a rational function $P_i(z)$. 
Moreover, the rational  function $P_i(z)$ is  regular at $z^{\pm1}=\infty$ and satisfies $P_i(0)P_i(\infty)=1$.

Denote $\mathbf{P}= (P_0(z),\ldots,P_{n-1}(z))$.
We call the $n$-tuple of rational functions $\bs P$ the highest weight of $V$.

Highest weight modules of $\E'_n$ were studied in detail in \cite{FJMM1}.

\medskip
Since the evaluation map is graded, the evaluation highest weight $U'_q\widehat{\gl}_n$ module with 
highest weight vector $v$ is 
a highest weight $\E'_n$ module with highest vector $v$. Indeed, using Lemma 2.4 in \cite{FJMM}, we find the principal degrees
$${\rm pdeg}\, \ev^{(3)}_u
(\theta^{-1}(E_{i,k}))=1, 
\qquad {\rm pdeg}\, \ev^{(3)}_u
(\theta^{-1}(H_{i,r}))=0, $$
 $0\leq i\leq n-1$, $k\in\Z$, $r\in \Z\setminus\{0\}$.

The following proposition describes the corresponding highest weight.

\begin{thm}\label{hw thm} 
Let $V$ be a highest weight $U'_q\widehat{\gl}_n$ module with highest weight 
$(\kappa_0,\dots,\kappa_{n-1})$. 
Let $V(u)$ be the evaluation  $\E^{(3)}_n$ module. Then 
$q_3^{n}=\prod_{i=0}^{n-1} \kappa_i^2$ 
and $V(u)$ is a highest weight module with highest weight
\begin{align*}
\bs P=\left(\kappa_0\frac{1-\kappa_0^{-2}u'/z}{1-u'/z},\ 
\kappa_{1}
\frac{1-q_3(\kappa_0\kappa_1)^{-2}u'/z}{1-q_3\kappa_0^{-2} u'/z},
\ \dots,\ 
\kappa_{n-1}\frac{1-q_3^{n-1}(\prod_{i=0}^{n-1}\kappa_i^{-2})u'/z}
{1-q_3^{n-1}(\prod_{i=0}^{n-2}\kappa_i^{-2})u'/z}
\right)\,,
\end{align*}
for an appropriate choice of $u'\in\C^\times$.
\end{thm}
\begin{proof}
The proposition is proved similarly to Theorem 5.7 in \cite{MY}.
\end{proof}
In the case of evaluation modules defined for $C=q_1^{n/2}$, 
the highest weight reads
\begin{align*}
\bs P=\left(\kappa_0\frac{1-u'/z}{1-\kappa_0^2u'/z},\ 
\kappa_1\frac{1-q_1^{-1}\kappa_0^2u'/z}{1-q_1^{-1}\kappa_0^2\kappa_1^2u'/z},
\ \dots,\ 
\kappa_{n-1}\frac{1-q_1^{-n+1}(\prod_{i=0}^{n-2}\kappa_i^2)u'/z}
{1-q_1^{-n+1}(\prod_{i=0}^{n-1}\kappa_i^2) u'/z}\right)\,.
\end{align*}

\medskip
 
It follows from Theorem \ref{hw thm} that the 
modules $\mc H^{(k)}(u_1,\dots,u_n)$ and 
$\mc G_{\mu, \nu}^{(k)}$ in \cite{FJMM1} are 
evaluation Verma and Weyl type modules respectively.

\subsection{Wakimoto evaluation modules}\label{GZ:sec}
A Gelfand-Zeitlin pattern for $\gl_n$ is an array of complex numbers 
\begin{align*}
\bla=
\begin{matrix}
\lambda_{0,n-1}&     &\lambda_{1,n-1} &    &  \cdots  & \lambda_{n-2,n-1}&     &           &\lambda_{n-1,n-1} \\
              &\lambda_{0,n-2}&      &\lambda_{1,n-2}& \cdots            &     &\lambda_{n-2,n-2}&     \\
              &              &\ddots &               &\ddots             &                 & &    \\
             &               &       &\lambda_{0,1}  &                   &\lambda_{1,1}    &  &   \\
             &               &       &               &\lambda_{0,0}      &                 & &    \\
\end{matrix}
\end{align*}

Fix generic complex numbers $\lambda^0_{i,j}\in\C$ and consider the linear space ${\rm GZ}_{\bla^0}$ with basis
$\{\ket{\bla}\}$, where $\bla$ runs over Gelfand-Zeitlin patterns for $\gl_n$
such that $\lambda_{i,j}\in \lambda^0_{i,j}+\Z$ for $0\le i\le j\le n-2$ 
and $\lambda_{i,n-1}=\lambda^0_{i,n-1}$ for $0\le i\le n-1$:
\begin{align}\label{adm part}
{\rm GZ}_{\bla^0}={\rm Span}_{\C}\{\ket{\bla}\mid 
\lambda_{i,j}\in \lambda^0_{i,j}+\Z\ (0\le i\le j\le n-2), ~ \lambda_{i,n-1}=\lambda^0_{i,n-1}\ (0\le i\le n-1)\}\,.
\end{align}

\begin{prop} \cite{J1},\cite{NT}
\label{GZ action}
The following formulas give a representation of $U_q\mathfrak{gl}_n$ on ${\rm GZ}_{\bla^0}$.
\begin{align*}
&e_i\ket{\bla}=\sum_{k=0}^{i-1}c_{k,i-1}^{+}(\bla)\ket{\bla+\one_{k,i-1}}\,,
\\
&f_i\ket{\bla}=\sum_{k=0}^{i-1}c_{k,i-1}^{-}(\bla)\ket{\bla-\one_{k,i-1}}\,,
\\
&q^{\varepsilon_r}\ket{\bla}=q^{h_{r}(\bla)-h_{r-1}(\bla)}\ket{\bla}\,.
\end{align*}
Here $1\le i\le n-1$,  $0\le r\le n-1$, 
$\one_{i,j}$ stands for the vector 
$(\delta_{r,i}\delta_{s,j})_{0\le r\le s\le n-1}$, 
and 
\begin{align*}
&c^+_{k,i-1}(\bla)=\frac{\prod_{l=0}^{i-2}[\lambda_{l,i-2}-\lambda_{k,i-1}-l+k-1]}
{\prod_{l=k+1}^{i-1}[\lambda_{l,i-1}-\lambda_{k,i-1}-l+k-1][\lambda_{l,i-1}-\lambda_{k,i-1}-l+k]}\,,
\\
&c^-_{k,i-1}(\bla)=-\frac{\prod_{l=0}^{i}[\lambda_{l,i}-\lambda_{k,i-1}-l+k+1]}
{\prod_{l=0}^{k-1}[\lambda_{l,i-1}-\lambda_{k,i-1}-l+k][\lambda_{l,i-1}-\lambda_{k,i-1}-l+k+1]}\,,
\\
&h_r(\bla)=\sum_{k=0}^{r}\lambda_{k,r}\,.
\end{align*}
\qed
\end{prop}
\bigskip
We call the $U_q\gl_n$ module ${\rm GZ}_{\bla^0}$ the  Gelfand-Zeitlin module.

Now we choose an arbitrary $\kappa \in\C^\times$ and 
define the $U_{q,\kappa}'\widehat{\mathfrak{gl}}_n$ module $\widehat{\rm GZ}_{\bla^0}$ of level $\kappa$ as follows.

Consider the $U_q\gl_n$ Gelfand-Zeitlin module ${\rm GZ}_{\bla^0}$ described in Proposition \ref{GZ action}. 
We trivially extend it to the action of the subalgebra ${\mathfrak b}_{q,\kappa}$ by setting
$$
x^{\pm}_{i,r}\ket{\bla}   
=h_{j,r}\ket{\bla}=0 \qquad (1\le i\le n-1, \quad 0\le j\le n-1 , \quad r\ge 1, \quad \ket{\bla}\in {\rm GZ}_{\bla^0}).$$ 
Then we induce
\begin{align*}
\widehat{\rm GZ}_{\bla^0}=
{\rm Ind}_{{\mathfrak b}_{q,\kappa}}^{U'_{q,\kappa}\widehat{\mathfrak{gl}}_n}\ 
{\rm GZ}_{\bla^0}.
\end{align*}

We call $\widehat{\rm GZ}_{\bla^0}$ the {Wakimoto module}.
Clearly, any Wakimoto module is an admissible $U_{q,\kappa}'\widehat{\mathfrak{gl}}_n$ module. 

Finally, we choose $u\in\C^\times$, set $q_3^{n/2}=\kappa$ and, using the 
evaluation homomorphism $\ev_u^{(3)}$, we view $\widehat{\rm GZ}_{\bla^0}$ as  
an $\E'_n(q_1,q_2,q_3)$ module. We denote this evaluation module by $\widehat{\rm GZ}_{\bla^0}(u)$. 

Recall the subalgebras $\E_{1,m}=\E_1(q_{1,m},q_2,q_{3,m})\subset \E'_n$, 
$0\le m\le n-1$, and the algebra $\mathcal A$, see \eqref{A}.

For each $i\in \{1,2,3\}$, the Fock module $\F_i(u)$ for $\E_1'$ is defined to be
the highest weight $\E'_1(q_1,q_2,q_3)$ module 
with highest weight $$\bs P=q_i^{1/2}\ \frac{1-q_i^{-1}u/z}{1-u/z}.$$ 

The Fock modules are well known and understood, see for example \cite{FFJMM}, \cite{FJMM1}. 
For $\E_1'$, 
the action of $E_0(z), F_0(z), K_0^\pm(z)$ can be described explicitly via one free boson and vertex operators. 
In particular the character $\sum_{r\ge 0}\dim (\F_i(u))_{-r} x^{r}$
of the Fock module $\F_i(u)$ is given by $1/(x)_\infty=\prod_{r>0}(1-x^r)^{-1}$.

Our main result of this section is the following statement.
For this statement we make a technical assumption that $q$ is generic, meaning we exclude a set of values of $q$ which is at most countable.
\begin{thm}\label{decompose}
We have the decomposition of 
$\mathcal{A}$ modules
\begin{align*}
\widehat{\rm GZ}_{\bla^0}(u)=\oplus_{\bla}\ W(\bla)\,,
\quad 
W(\bla)=\mathcal{A}\ket{\bla}=
W_0(\bla)\boxtimes W_1(\bla)\boxtimes\cdots\boxtimes W_{n-1}(\bla)\,,
\end{align*}
where the sum is over all GZ patterns  as in \eqref{adm part}, 
and the $\E_{1,m}$ module $W_m(\bla)$ is given by 
\begin{align*}
&W_m(\bla)=
\overbrace{\F_3(u_{0,m}(\bla))\otimes\cdots\otimes\F_3(u_{m,m}(\bla))}^{m+1}
\otimes 
\overbrace{\F_1(v_{0,m-1}(\bla))\otimes\cdots \otimes \F_1(v_{m-1,m-1}(\bla))}^{m}
\,,
\end{align*}
where 
\begin{align*}
&u_{l,m}(\bla)=q_2^{\lambda_{l,m}-l+m}q_3^n\tilde{u}\,,    
\quad v_{l,m-1}(\bla)=q_2^{\lambda_{l,m-1}-l-1}\tilde{u}\,, \\
&
\tilde{u}=(-1)^{n}
q^{-\sum_{r=0}^{n-1}\lambda^0_{r,n-1}+n-1}u\,.
\end{align*}
\end{thm}

Theorem \ref{decompose} is proved in Appendix \ref{sec: Appendix}.

The $\E_1$ modules $W_m(\bla)$ also appeared in \cite{BFM}.

\appendix
\section{}\label{sec: Appendix}
We sketch the proof of Theorem \ref{decompose}.
\subsection{The plan of the proof}
Our logic is the following. 

Denote $K_i^{\pm,=}(z)=\theta^{-1}(K_i(z))$. 
Denote $\ssK_m^{\pm, =}(z)=\theta^{-1}_m(\ssK_m(z))$, where $\theta_m$ is the Miki automorphism of $\E_{1,m}$.

Since the degrees of $K_i^{\pm,=}(z),\ssK_m^{\pm,=}(z)$ are zero, these operators preserve the space 
${\rm GZ}_{\bla^0}\subset \widehat{\rm GZ}_{\bla^0}$. We show that these operators are diagonal in the basis $\{\ket{\bla}\}$
and give their eigenvalues. 

First, we compute the eigenvalues of $K_i^{\pm,=}(z)$ using the fact that these operators belong to the horizontal algebra,
which acts on ${\rm GZ}_{\bla^0}$ through the standard evaluation map $\overline \ev_u$.
Then we calculate the eigenvalues of $\ssK_m^{\pm, =}(z)$ by finding the projection of these operators to the algebra 
generated by $K_i^{\pm,=}(z)$ along annihilating operators. This is a long calculation. 
We first compute the projection of the first components of $\ssK_m^{\pm, =}(z)$ explicitly and then argue that this is sufficient.

We observe that the eigenvalue of $\ssK_m^{\pm,=}(z)$ on eigenvector $\ket{\bla}$ coincides with the highest weight of module $W_m(\bla)$. 
It follows that the character of $\mathcal A \ket{\bla}$ is at least 
$((x)_\infty)^{-n^2}$. We show that the joint spectrum of $\ssK_m^{\pm,=}(z)$, $0\le m\le n-1$, is simple in 
$\oplus_{\bla} W(\bla)$.

Comparing to the character of $\widehat{\rm GZ}_{\bla^0}$, we obtain the theorem.

\subsection{Projection}\label{sec:projection}
Let $\E_n^\pm \subset \E_n'$ 
be the subalgebras generated by $E_{i,r},F_{i,r},H_{i,r}$ with $0\le i\le n-1$ and $\pm r>0$,
and let $\E^0_n=U_q^{h}\widehat{\gl_n}$. Let $\widetilde \E_n^+ \subset \widetilde \E_n'$ be the completion of $\E_n^+$.
We use the triangluar decomposition 
\begin{align*}
\widetilde \E_n'=  \E^{-}_n\otimes \E_n^0 \otimes \widetilde \E_n^+\,.
\end{align*}
While we expect this to hold for all $q$, we have been unable to find it in the literature.  
For generic $q$ it can be shown by taking the limit $q\to 1$ and using the result of \cite{N}.

Consider the projection to the middle factor in the triangular decomposition
\begin{align*}
 \pr:\widetilde \E'_n \longrightarrow \E^0_n\,.
\end{align*}
We shall write $x\equiv y$ if $\pr(x)=\pr(y)$ for $x,y\in\widetilde \E'_n$. 
Similarly to the usual Harish-Chandra map, this projection is a homomorphism when 
restricted to the subalgebra $(\widetilde \E'_n)_0\subset \widetilde \E_n'$ consisting of elements of homogeneous degree $0$.

\iffalse
\begin{lem}\label{gl red} We have $\pr (\widetilde \E_n^0) \subset U_q^{h}\widehat{\gl_n}$. 
\end{lem}
\begin{proof} Rotate our picture by $\theta$. Then our degree becomes $-\deg(F_{i,s})=\deg(E_{i,s})=\delta_{0,i}$.
It is easy to see that the result of the projection can be always placed in the algebra which does not contain $E_{0,s}$, $F_{0,s}$, that is in the vertical algebra $U_q^{v}\widehat{\gl_n}$.
\end{proof}
\fi

The algebras $\E_{1,m}$ are obtained by taking $(n-1)$ simple fusions. 
It is sufficient to consider a single simple fusion. For that purpose, 
following \cite{FJMM1}, we consider the subalgebra
\begin{align*}
\E^{0|1}_{n-1}=
\E'_{n-1}(\tilde q_1,q_2,\tilde q_3)\ \subset \ \widetilde \E'_n(q_1,q_2,q_3)\,,
\quad \tilde q_1=q_1q_1^{\frac{1}{n-1}}\,,  \tilde q_3=q_3q_1^{-\frac{1}{n-1}}\,.
\end{align*} 
It is generated by the following currents
\begin{align*}
&\tilde E_i(z)=E_{i+1}(q_1^{\frac{i}{n-1}}z)\,,\quad
\tilde F_i(z)=F_{i+1}(q_1^{\frac{i}{n-1}}z)\,,\quad
\tilde K^\pm_i(z)=K^\pm_{i+1}(q_1^{\frac{i}{n-1}}z)
\quad (1\le i\le n-2)\,,
\\
&\tilde{E}_0(z)=E_{0|1}(z)=\Bigl(1-\frac{z}{z'}\Bigr)
E_{0}(q_1z')E_1(z)\Bigl|_{z'=z}\,,
\\
&\tilde{F}_0(z)=F_{0|1}(z)
=\Bigl(1-\frac{z'}{z}\Bigr)F_1(z)F_{0}(q_1z')\Bigl|_{z'=z}\,,
\\
&\tilde K^\pm_0(z)=K^\pm_0(q_1z)K_1^\pm(z)\,,
\end{align*}
and for $n=2$ 
\begin{align*}
&-\frac{qq_1^{-1}q_3}{1-q_1^{-1}q_3}
\tilde{E}_0(z)=E^{(1)}_{0|1}(z)
=\Bigl(1-\frac{z}{z'}\Bigr)
E_{0}(q_1z')E_1(z)\Bigl|_{z'=z}\,,
\\
&(1-q_1^2)\tilde{F}_0(z)=F^{(1)}_{0|1}(z)
=\Bigl(1-\frac{z'}{z}\Bigr)
F_1(z)F_{0}(q_1z')\Bigl|_{z'=z}\,,
\\
&\tilde K^\pm_0(z)=K^\pm_0(q_1z)K_1^\pm(z)\,.
\end{align*}
We also set $\tilde K_i^\pm(z)=\tilde K_i^{\pm1}\exp(\pm(q-q^{-1})\sum_{r>0} \tilde H_{i,\pm r}z^{\mp1})$.

We will use also another subalgebra $\E^{0|n-1}_{n-1}$
obtained from $\E^{0|1}_{n-1}$ by applying the isomorphism $\iota$ \eqref{iota}.
In what follows we study projections of elements of $\E^{0|1}_{n-1}$, $\E^{0|n-1}_{n-1}$.
We will write formulas only for the former, those for the latter are easily obtained by applying
$\iota$.

Let $\cN_{\ge r}$ denote the left ideal of 
$\widetilde\E_n'$ generated by elements of homogeneous degree
$\ge r$. 
\begin{lem}\label{EFtop}
Let $n\ge3$. Then
\begin{align}
&\tilde E_{0,0}   %E_{0|1,0}
\equiv -q[E_{1,0},E_{0,0}]_{q^{-1}} \bmod \cN_{\ge1}\,,
\quad
\tilde F_{0,0}   %F_{0|1,0}
\equiv -q^{-1}[F_{0,0},F_{1,0}]_{q}  \bmod \cN_{\ge1}\,,
\label{EFfus0}
\\
&\tilde E_{0,1}  %E_{0|1,1}
\equiv q_1^{-1}[E_{0,1},E_{1,0}]_{q}  \bmod \cN_{\ge2}\,,
\quad
\tilde F_{0,-1}  %F_{0|1,-1}
\equiv -q_1q^{-1}[F_{0,-1},F_{1,0}]_{q}  \bmod \cN_{\ge1}\,.
\label{EFfus1}
\end{align} 
For $n=2$, setting $a=q(1-q_1^2)(1-q_1^{-1}q_3)$ and $b=-qq_1^2$ we have
\begin{align*}
&b^{-1}\tilde E_{0,0}   %E^{(1)}_{0|1,0}
\equiv E_{0,0}E_{1,0}-(1+q_1^{-1}q_3-q_1^{-2})E_{1,0}E_{0,0}
+q_1^{-1}E_{1,1}E_{0,-1}\bmod \cN_{\ge1}\,,
\\
&b^{-1}\tilde E_{0,1}  %E^{(1)}_{0|1,1}
\equiv q_1^{-1}E_{0,1}E_{1,0}
-q_1^{-2}(q_1+q_3-q_1^{-1})E_{1,0}E_{0,1}
-q_1^{-1}q_3E_{0,0}E_{1,0}
+q_1^{-2}E_{1,1}E_{0,0}\bmod \cN_{\ge2}\,,
\\
&ab \tilde F_{0,0}   %F^{(1)}_{0|1,0}
\equiv 
F_{1,0}F_{0,0}
-(1+q_1q_3^{-1}-q_1^{2})F_{0,0}F_{1,0}
+q_1 F_{0,1}F_{1,-1}\bmod \cN_{\ge1}\,,
\\
&ab \tilde F_{0,-1}   %F^{(1)}_{0|1,-1}
\equiv 
q_1F_{1,0}F_{0,-1}
-q_1^2(q_1^{-1}+q_3^{-1}-q_1)F_{0,-1}F_{1,0}
-q_1q_3^{-1}F_{1,-1}F_{0,0}
+q_1^2 F_{0,0}F_{1,-1}\bmod \cN_{\ge1}\,.
\end{align*} 
\end{lem}
\begin{proof}
Let $n\ge3$. By the definition along with \eqref{quad-rel} we obtain
\begin{align*}
E_{0|1,l}
&=\sum_{j\ge0}
q_1^{j-l}\Bigl(
E_{0,-j+l}E_{1,j}-q_1 E_{0,-j-1+l}E_{1,j+1}
\Bigr) \\
&+\sum_{j\ge1} q^{-1}q_1^{-j-l}\Bigl(
E_{1,-j}E_{0,j+l}-q^{-1}_3 E_{1,-j+1}E_{0,j+l-1}\Bigr)\,, \\
F_{0|1,l}
&=\sum_{j\ge0}
q_1^{-j-l}\Bigl(
F_{1,-j}F_{0,l+j}-q^{-1}_1 F_{1,-j-1}F_{0,l+j+1}
\Bigr) \\
&+\sum_{j\ge1} qq_1^{j-l}\Bigl(
F_{0,l-j}F_{1,j}-q_3 F_{0,l-j+1}F_{1,j-1}
\Bigr)\,,
\end{align*} 
from which the assertion follows.

For $n=2$, we have
\begin{align*} 
E^{(1)}_{0|1,l}=\sum_{j\in\Z}q_1^{j-l}\bigl(E_{0,-j+l}E_{1,j} -q_1^{-1}E_{0,-j+l-1}E_{1,j+1}\bigr)\,.
\end{align*}
Rewriting it as 
\begin{align*}
(1-q_1^{-1}q_3)E^{(1)}_{0|1,l}=\sum_{j\in\Z}q_1^{j-l}
\bigl(E_{0,-j+l}E_{1,j}-(q_1+q_3)E_{0,-j+l-1}E_{1,j+1}+q_1q_3
E_{0,-j+l+2}E_{1,j+2}\bigr)
\end{align*}
and applying  \eqref{quad-rel2} we obtain 
\begin{align*}
(1-q_1^{-1}q_3)E^{(1)}_{0|1,l}
&=\sum_{j\ge0}
q_1^{j-l}\Bigl(
E_{0,l-j}E_{1,j}-(q_1+q_3)E_{0,l-j-1}E_{1,j+1}
+q_1q_3 E_{0,l-j-2}E_{1,j+2}\Bigr)
\\
&+\sum_{j\ge1} q_1^{-j-l}
\Bigl(q_1q_3 E_{1,-j}E_{0,j+l}
-(q_1+q_3)E_{1,-j+1}E_{0,l+j-1}
+ E_{1,-j+2}E_{0,j+l-2}\Bigr)
\,,
\end{align*}
Proceeding similarly with  $F^{(1)}_{0|1}(z)$ we get
\begin{align*} 
(1-q_1q_3^{-1})F^{(1)}_{0|1,l}
&=\sum_{j\ge0}
q_1^{-j-l}\Bigl(
F_{1,-j}F_{0,l+j}-
(q_1^{-1}+q_3^{-1})F_{1,-j-1}F_{0,l+j+1}
+q_1^{-1}q_3^{-1}F_{1,-j-2}F_{0,l+j+2}
\Bigr) \\
&+\sum_{j\ge1} q_1^{j-l}\Bigl(
q_1^{-1}q_3^{-1}F_{0,l-j}F_{1,j}
-(q_1^{-1}+q_3^{-1})F_{0,l-j+1}F_{1,j-1}
+
F_{0,l-j+2}F_{1,j-2}
\Bigr)\,,
\end{align*} 
which imply the relations stated in Lemma.
\end{proof}

\subsection{Action of $K_i^{\pm,=}(z)$ on ${\rm GZ}_{\bla^0}$}

Recall that the horizontal subalgebra 
 $U^h_q\widehat{\sln}_n$ 
of $\E_n'$ is given in Chevalley generators. The next lemma
describes the corresponding Drinfeld generators under the identification \eqref{Chev}.

Note that we use a different identification \eqref{left Chev} for Chevalley generators of the vertical algebra $U^v_q\widehat{\sln}_n$.

\begin{lem}\label{K=phi}
The Drinfeld generators of the horizontal subalgebra $U^h_q\widehat{\sln}_n$ are given by
\begin{align*}
&x_i^+(z)=\theta^{-1}(E_{i}(d^{-i}z))\,,\quad 
x_i^-(z)=\theta^{-1}(F_{i}(d^{-i}z))\,,\\
&
\phi^\pm_i(z)=\theta^{-1}(K^\pm_{i}(d^{-i}z))
\quad (1\le i\le n-1). 
\end{align*}
\end{lem}
\begin{proof}
Define the currents $x^\pm_i(z),\phi^\pm_i(z)$ by the above formulas. 
Then they belong to the horizontal 
subalgebra and satisfy the relations for the Drinfeld currents of 
$U'_q\widehat{\sln}_n$.
We have $x^+_{i,0}=E_{i,0}, x^-_{i,0}=F_{i,0}$ for $1\le i\le n-1$.
Moreover, from \eqref{thetaE} and \eqref{thetaF} we obtain 
\begin{align*}
&E_{0,0}=[x^-_{n-1,0},\ldots,[x^-_{2,0},x^-_{1,-1}]_q\ldots]_qK_1\cdots K_{n-1}\,,
\\ 
&F_{0,0}=(K_1\cdots K_{n-1})^{-1}[\ldots [x^+_{1,1},x^+_{2,0}]_{q^{-1}},\ldots,x^+_{n-1,0}]_{q^{-1}}\,.
\end{align*}
Hence the assertion follows from the identification \eqref{Chev}
 between the Chevalley and the Drinfeld generators.
\end{proof}

Recall the standard evaluation map 
$\overline{\ev}_u$, \eqref{ev1}.
%\eqref{ev0}-\eqref{ev2}. 
\begin{lem}\label{top ev}
On ${\rm GZ}_{\bla^0}$, the action of horizontal algebra $U_q^h\widehat{\mathfrak{sl}}_n$ is given by $\overline{\ev}_u$. Namely, 
 for any $x\in U_q^h\widehat{\mathfrak{sl}}_n$, we have
$$
\ev^{(3)}_u(x)\ket{\bla}=\overline{\ev}_{u}(x)\ket{\bla}\,.
%\quad \bar u=q^{-\sum_{i=0}^{n-1}\lambda^0_{i,n-1}}u\,.
%{\sf t }^{-1}u\,.
$$ 
\end{lem}
\begin{proof}
It suffices to check the statement for $x=e_0,f_0$.  
By the definition of $\ev^{(3)}_u$ we have
\begin{align*}
\ev^{(3)}_u
\bigl(E_0(z)\bigr)\ket{\bla}
=u^{-1} 
\mathcal{K}e^{A_-(z)} \ssF(z)\ket{\bla}\,,
\end{align*}
and in particular 
\begin{align*}
\ev^{(3)}_u 
(E_{0,0})\ket{\bla}=u^{-1} 
\mathcal{K}\ssF_0\ket{\bla}\,.
\end{align*}
Using the $q_3$ version of Lemma \ref{EFtop}
repeatedly we find that the right hand side becomes
\begin{align*}
u^{-1} q^{-\Lambda_1+\Lambda_{n-1}}
[\ldots [F_{1,0},F_{2,0}]_{q^{-1}},\ldots, F_{n-1,0}]_{q^{-1}}\ket{\bla}
\end{align*}
and hence coincides with $\overline{\ev}_{u}(e_0)\ket{\bla}$.

The case of $f_0$ is entirely similar.
\end{proof} 
From Lemma \ref{top ev}, we obtain 
 the explicit action of Drinfeld generators of horizontal algebra $U_q^h\widehat{\mathfrak{sl}}_n$ on the 
``top level'' ${\rm GZ}_{\bla^0}$.

\begin{lem}\label{lem:ev3GZ0}
On ${\rm GZ}_{\bla^0}$, the Drinfeld generators 
of horizontal algebra $U_q^h\widehat{\mathfrak{sl}}_n$ act as follows.
\begin{align*}
&
\ev^{(3)}_u\bigl(x_i^\pm(z)\bigr)
\ket{\bla}=\sum_{k=0}^{i-1}c_{k,i-1}^{\pm}(\bla)
\delta\bigl(q^{2\lambda_{k,i-1}+i-2k\pm1}{\bar u}/z
\bigr)\ket{\bla\pm\one_{k,i-1}}\,,
\\
&\ev^{(3)}_u\bigl(\phi_i^\pm(z)\bigr)
\ket{\bla}=q^{2h_{i-1}(\bla)-h_{i-2}(\bla)-h_i(\bla)}
\frac{\prod_{l=0}^{i-2}(1-q^{2\la_{l,i-2}+i-2l-1}{\bar u}/z)
\prod_{l=0}^{i}(1-q^{2\la_{l,i}+i-2l+1}{\bar u}/z)}
{\prod_{l=0}^{i-1}(1-q^{2\la_{l,i-1}+i-2l-1}{\bar u}/z
)(1-q^{2\la_{l,i-1}+i-2l+1}{\bar u}/z
)}\ket{\bla}\,,
\end{align*}
where $1\le i\le n-1$ and 
\begin{align}
\bar u=(-1)^nq^{-\sum_{r=0}^{n-1}\lambda^0_{r,n-1}+n-2}u \,.
\label{ubar}
\end{align}
\end{lem}
\begin{proof}
Define operators
$\ev'_{\bar u}\bigl(x_i^\pm(z)\bigr)$, 
$\ev'_{\bar u}\bigl(\phi_i^\pm(z)\bigr)$
acting on $\ket{\bla}$ by 
the right hand sides of the above formulas. 
A direct computation shows that they satisfy the defining relations for 
 $U_q'\widehat{\mathfrak{sl}}_n$. It suffices to check that 
$\ev'_{\bar u}(x)=\ev^{(3)}_u(x)$ for $x=e_0,f_0$.

We have 
\begin{align*}
\ev'_{\bar u}(x_{1,-1}^-)\ket{\bla}
={\bar u}^{-1}\, f_1q^{-2\varepsilon_0}\ket{\bla}\,,
\quad 
\ev'_{\bar u}(x^-_{i,0})\ket{\bla}=f_i\ket{\bla}\quad (2\le i\le n-1).
\end{align*}
Noting that $q^{\varepsilon_0}$ commutes with $f_i$ for $i\ge2$
and using \eqref{Chev}, 
we obtain
\begin{align*}
\ev'_{\bar u}(e_0)
&=\ev'_{\bar u}\bigl([x^-_{n-1,0},\ldots,[x^-_{2,0},x^-_{1,-1}]_q\ldots]_q
q^{\alpha_1+\cdots+\alpha_{n-1}}\bigr)\ket{\bla}
\\
&={\bar u}^{-1}[f_{n-1},\ldots,[f_2,f_1]_q\ldots]_q
q^{-2\varepsilon_0}q^{\varepsilon_0-\varepsilon_{n-1}}\ket{\bla}
\\
&=(-q)^{n-2}{\bar u}^{-1}[\ldots[f_1,f_2]_{q^{-1}},\ldots,f_{n-1}]_{q^{-1}}
q^{-\Lambda_1+\Lambda_{n-1}}{\sf t}^{-1}\ket{\bla}
\\
&=\ev^{(3)}_u(e_0)\,.
\end{align*}
%It follows that the 
%$e_0$ action is given by the right hand side of \eqref{ev1}. 
Similarly we check %\eqref{ev2}.
$\ev'_{\bar u}(f_0)=\ev^{(3)}_u(f_0)$.
\end{proof}

Now we are in a 
position to compute the
action of $\bar K^{\pm,=}_i(z)=K_i^{\mp1}K_i^{\pm,=}(z)$ on ${\rm GZ}_{\bla^0}$.

Denote by $\bar K_i(z,\bla)$ 
the eigenvalues of $\bar K^{\pm,=}_i(z)$ on $\ket{\bla}$, $0\le i\le n-1$. 

\begin{prop} \label{K^= action}
With the definition \eqref{ubar}, we have
\begin{align*}
&\bar K_i(z,\bla)
=\frac{\prod_{l=0}^{i-2}(1-q_3^i q_2^{\la_{l,i-2}+i-l-1}
q \bar u/z)   
\prod_{l=0}^{i}(1-q_3^i q_2^{\la_{l,i}+i-l} q \bar u/z 
)}
{\prod_{l=0}^{i-1}(1-q_3^iq_2^{\la_{l,i-1}+i-l-1} q \bar u/z  
)(1-q_3^iq_2^{\la_{l,i-1}+i-l} q \bar u/z  
)}
\quad (1\le i\le n-1),
\\
&\bar K_0(z,\bla)
=\frac{1-q_2^{\lambda_{0,0}} q\bar u/z  
}{1-q_3^n q_2^{\lambda_{n-1,n-1}}q \bar u/z 
}
\prod_{l=0}^{n-2}\frac{1-q_3^n q_2^{\la_{l,n-2}+n-l-1}q\bar u/z  
}{1-q_3^n q_2^{\la_{l,n-1}+n-l-1}q\bar u/z  
}
\,.
\end{align*}
\end{prop}
\begin{proof}
The formulas for $1\le i\le  n-1$ follow from Lemmas \ref{K=phi} and \ref{lem:ev3GZ0}

Consider the special case where $\bla$ is dominant. 
Then the Wakimoto module has a highest weight submodule, 
and the eigenvalue of $K_0^{\pm,=}(z)$ 
on highest weight vector can be determined by Theorem \ref{hw thm} 
along with the knowledge of $\bar K_i(z,\bla)$ for $1\le i\le  n-1$. 
In this case general eigenvalues can then be obtained by acting with 
$E_i(z), F_i(z)$. 

Since the eigenvalues of $K^\pm_{i,r}$ are polynomial 
functions of the parameters $q^{\lambda_{j,r}}$,  
the formula for  $K_0^{\pm,=}(z)$ in the general case 
follows by ``analytic continuation''.
\end{proof}

Denote by $\Psi_m(z,\bla)$ the highest weight of $\E_{1,m}$ module $W_m(\bla)$ and let 
$\bar\Psi_m(z,\bla)=\Psi_m(z,\bla)/\Psi_m(\infty,\bla)$.
\begin{cor}\label{K act}
For $0\le m\le n-1$, the highest weight of $\E_{1,m}$ module $W_m(\bla)$ is given by
\begin{align*}
\bar\Psi_m(z,\bla)&=\prod_{i=m+1}^{n-1}\bar K_i(q_3^{-n+i}z,\bla)
\prod_{i=0}^m \bar K_i(q_1^{-i}z,\bla)\\
&=
\prod_{l=0}^m \frac{1-q_2^{\lambda_{l,m}-l}q \bar u/z 
}{1-q_2^{\lambda_{l,m}-l+m}q_3^n q\bar u/z  
}
\prod_{l=0}^{m-1} \frac{1-q_2^{\lambda_{l,m-1}-l+m}q_3^n q\bar u/z  
}{1-q_2^{\lambda_{l,m-1}-l-1}q\bar u/z  
},
\end{align*}
where $\bar u$ is given by \eqref{ubar}. 
\qed
\end{cor}

\subsection{Action of $\ssK_{m}^{\pm, =}(z)$ on ${\rm GZ}_{\bla^0}$: the first components}
We show that the action of $\ssK_m^{\pm,=}(z)$ on the top level is diagonal in the basis of 
$\ket{\bla}$ with eigenvalues given in Corollary \ref{K act}. 
More precisely, we show  the following formulas for the horizontal currents:
\begin{align}\label{projection}
\ssK_m^{\pm,=}(z)\equiv \prod_{i=m+1}^{n-1}\bar K_i^{\pm,=}(q_3^{-n+i}z)
\prod_{i=0}^m \bar K_i^{\pm,=}(q_1^{-i}z).
\end{align}

In this section we do it for the 
first components using an explicit computation.

Recall the subalgebra $\E^{0|1}_{n-1}$ and the currents $\tilde K_i^\pm(z)$ in Section 
\ref{sec:projection}.
Let $\tilde\theta$ be the Miki automorphism for $\E^{0|1}_{n-1}$.

To show \eqref{projection}, it is enough to prove the following statement.
\begin{prop}\label{simple proj}
For all $0\le i \le n-2$ we have
\begin{align}
&\tilde\theta^{-1}(\tilde K^\pm_i(z))\equiv
\theta^{-1}\bigl(K^\pm_{i+1}(q_1^{\frac{i}{n-1}-1}z)\bigr)
\quad (1\le i\le n-2),
\label{KKK1}\\ 
&\tilde\theta^{-1}(\tilde K^\pm_0(z))\equiv 
\theta^{-1}\bigl(K^\pm_{0}(z)K^\pm_{1}(q_1^{-1}z)\bigr)\,.
\label{KKK2}
\end{align}
\end{prop}

In this section first we explicitly compute the projections of $\tilde{\theta}^{-1}(\tilde H_{i,\pm1})$. 
The general case is done in Section \ref{general K sec}.

\begin{prop}\label{Hi1-n}
Let $n\ge3$. We have $\bmod$ $\cN_{\ge1}$,
\begin{align*}
&\tilde\theta^{-1}\bigl(\tilde H_{i,1}\bigr)
\equiv  q_1^{1-\frac{i}{n-1}} \,\theta^{-1}\bigl( H_{i+1,1}\bigr)\,,
\quad 
\tilde\theta^{-1}\bigl(\tilde H_{i,-1}\bigr)
\equiv q_1^{-1+\frac{i}{n-1}} \, \theta^{-1}\bigl(H_{i+1,-1}\bigr)\,
\quad (1\le i\le n-2),
\\
&\tilde\theta^{-1}\bigl(\tilde H_{0,1}\bigr)
\equiv \theta^{-1}\bigl(H_{0,1}\bigr)+ q_1\theta^{-1}\bigl(H_{1,1}\bigr)\,,
\quad 
\tilde\theta^{-1}\bigl(\tilde H_{0,-1}\bigr)
\equiv 
\theta^{-1}\bigl(H_{0,-1}\bigr)+ q_1^{-1}\theta^{-1}\bigl( H_{1,-1}\bigr)
\,.
\end{align*}
\end{prop}
\begin{proof}
We prove the case $\tilde \theta^{-1}\bigl(H_{i,1}\bigr)$.  Set 
$\tilde d=q_1^{\frac{1}{n-1}}d$. 

Suppose $1\le i\le n-2$. 
By the definition
\begin{align*}
&\tilde\theta^{-1}\bigl(\tilde H_{i,1}\bigr)
=
-(-\tilde d)^{-i}
[[\cdots[[\cdots [\tilde F_{0,0},\tilde F_{n-2,0}]_{q},
\cdots,\tilde F_{i+1,0}]_{q},
\tilde F_{1,0}]_{q},
\cdots,\tilde F_{i-1,0}]_{q},
\tilde F_{i,0}]_{q^{2}}
\\
&=
-(-\tilde d)^{-i}
[[\cdots[[\cdots [\tilde F_{0,0},\tilde F_{n-2,0}]_{q},
\cdots,\tilde F_{i+1,0}]_{q},
\tilde F_{1,0}]_{q},
\cdots,\tilde F_{i-1,0}]_{q},
\tilde F_{i,0}]_{q^{2}}\,.
\end{align*}
Substituting \eqref{EFfus0} and noting that
\begin{align*}
[\cdots [F_{0|1,0},\tilde F_{n-2,0}]_q\cdots\tilde F_{i+1,0}]_q 
&\equiv-q^{-1}[\cdots [[F_{0,0},F_{1,0}]_q,
F_{n-1,0}]_q\cdots F_{i+2,0}]_q 
\\
&=
-q^{-1} [\cdots [[F_{0,0},F_{n-1,0}]_q\cdots 
F_{i+2,0}]_q F_{1,0}]_q 
\end{align*}
we obtain the result.

Next let $i=0$. From \eqref{H01} we have
\begin{align*}
\tilde\theta^{-1}\bigl(\tilde H_{0,1}\bigr)
&\equiv -(-\tilde d)^{-n+2}
[\cdots [\tilde F_{1,1},\tilde F_{2,0}]_{q},\cdots,
\tilde F_{n-2,0}]_{q}\cdot
\tilde F_{0,-1}
\,\\
&= -(-d)^{-n+1}q_1
A\cdot\bigl(F_{0,-1}F_{1,0}-q F_{1,0}F_{0,-1}\bigr)
\,,
\end{align*}
with
\begin{align*}
A= 
[[\cdots [F_{2,1},F_{3,0}]_{q},\cdots,F_{n-2,0}]_q,F_{n-1,0}]_{q}\,.
\end{align*}
Since $A$ has degree $1$, we have
\begin{align*}
A\cdot\bigl(F_{0,-1}F_{1,0}-q F_{1,0}F_{0,-1}\bigr) 
&\equiv [A,F_{0,-1}]_qF_{1,0}-q\bigl([A,F_{1,0}]_qF_{0,-1}+qF_{1,0}AF_{0,-1}
\bigr)\\
&\equiv
[[A,F_{0,-1}]_q,F_{1,0}]_{q^2}-q[[A,F_{1,0}]_q,F_{0,-1}]_{q^2}\,.
\end{align*}
Noting that $[[A,B]_p,C]_p=[[A,C]_p,B]_p$ if $BC=CB$, we find 
\begin{align*}
[A, F_{1,0}]_{q}
&=
[[\cdots [[F_{2,1},F_{1,0}]_q,F_{3,0}]_{q},\cdots,F_{n-2,0}]_q,F_{n-1,0}]_{q}\,
\\
&=(-d)^{-1}[[\cdots [[F_{1,1},F_{2,0}]_{q},F_{3,0}]_{q},\cdots,F_{n-2,0}]_q,F_{n-1,0}]_{q}\,.
\end{align*}
In the last line we use the quadratic relations.
By the same token we have
\begin{align*}
A&=(-d)[[\cdots [F_{3,1},F_{2,0}]_{q},\cdots,F_{n-2,0}]_q,F_{n-1,0}]_{q} \\
&=(-d)[[[\cdots [F_{3,1},F_{4,0}]_{q},\cdots,F_{n-2,0}]_q,F_{n-1,0}]_{q},F_{2,0}]_{q} \,,
\end{align*}
and repeating this we get
\begin{align*}
A&=(-d)^{n-3}[[\cdots [F_{n-1,1},F_{n-2,0}]_{q},\cdots,F_{3,0}]_{q},F_{2,0}]_{q}\,.
\end{align*}
Hence
\begin{align*}
[A, F_{0,-1}]_{q}
&\equiv (-d)^{n-3}[[\cdots [F_{n-1,1},F_{n-2,0}]_{q},\cdots,F_{2,0}]_{q},
F_{0,-1}]_q\\
&=(-d)^{n-3}[\cdots [[F_{n-1,1},F_{0,-1}]_q,
F_{n-2,0}]_{q},\cdots,F_{2,0}]_{q}\\
&=(-d)^{n-2}[\cdots [[F_{0,0},F_{n-1,0}]_q,
F_{n-2,0}]_{q},\cdots,F_{2,0}]_{q}\,.
\end{align*}
We obtain after simplification
\begin{align*}
\tilde\theta^{-1}\bigl(\tilde H_{0,1} \bigr)
&=\theta^{-1}(H_{0,1})+q_1\theta^{-1}(H_{1,1})\,.
\end{align*}

The case $\tilde H_{0,-1}$ can be handled in the same way.
\end{proof}

\begin{prop}\label{Hi1-2}
Let $n=2$. Then we have $\bmod$ $\cN_{\ge1}$
\begin{align*}
&\tilde{\theta}^{-1}(\tilde H_{0,1})
\equiv \theta^{-1}(H_{0,1})+q_1 \theta^{-1}(H_{1,1})\,,
\quad 
\tilde{\theta}^{-1}(\tilde H_{0,-1})
\equiv \theta^{-1}(H_{0,-1})+q_1^{-1} \theta^{-1}(H_{1,-1})\,.
\end{align*}
\end{prop}
\begin{proof}
We recall that the Miki automorphism for $\E^{0|1}_1$ is given by
\begin{align*}
&\tilde\theta^{-1}\bigl(\tilde H_{0,1}\bigr)=-a\tilde K_0^{-1}\tilde F_{0,0}\,,\quad
\tilde\theta^{-1}\bigl(\tilde H_{0,-1}\bigr)=-\tilde K_0\tilde E_{0,0}\,,\\
&\tilde\theta^{-1}\bigl(\tilde E_{0,0}\bigr)=-\tilde C^{-1}\tilde H_{0,1}\,,\quad
\tilde\theta^{-1}\bigl(\tilde F_{0,0}\bigr)=-a^{-1}\tilde C\tilde H_{0,-1}\,,\\
&\tilde\theta^{-1}\bigl(\tilde C\bigr)=\tilde K_{0}^{-1}\,,\quad
\tilde\theta^{-1}\bigl(\tilde K_0\bigr)=\tilde C\,,
\end{align*}
where $a=q(1-\tilde q_1)(1-\tilde q_3)$. 
As an example let us check $\tilde\theta^{-1}\bigl(\tilde H_{0,1}\bigr)$. 
By Lemma \ref{EFtop} we have
\begin{align*}
qq_1^2\tilde\theta^{-1}\bigl(\tilde H_{0,1}\bigr)%=F^{(1)}_{0|1,0}
\equiv F_{1,0}F_{0,0}-(1+q_1q_3^{-1}-q_1^{2})F_{0,0}F_{1,0}+q_1F_{0,1}F_{1,-1}\,.
\end{align*}
Using the quadratic relation
\begin{align*}
[F_{0,1},F_{1,-1}]_{q^{2}}=
[F_{1,1},F_{0,-1}]_{q^{2}}-(q_1^{-1}+q_3^{-1})[F_{1,0},F_{0,0}]\,,
\end{align*}
we obtain
\begin{align*}
qq_1^2\tilde\theta^{-1}\bigl(\tilde H_{0,1}\bigr)
&\equiv q_1\bigl([F_{1,1},F_{0,-1}]_{q^{2}}+q_1[F_{0,0},F_{1,0}]_{q^{2}}\bigr)\\
&=qq_1^2\bigl(\theta^{-1}(H_{0,1})+q_1\theta^{-1}(H_{1,1})\bigr)\,.
\end{align*}
\end{proof}

\subsection{Action of $\ssK_m^{\pm,=}(z)$ on ${\rm GZ}_{\bla^0}$: the general case}\label{general K sec}
To compute projections of  $\tilde\theta^{-1}(\tilde H_{i,\pm r})$ for $r\ge2$ we use the following argument.

Let us summarize our knowledge so far. Inside $\widetilde{\E}'_n$
we have mutually commuting subalgebras $\E'_{1,0}$ and $\E'_{n-1}$:
\begin{align*}
\E'_{1,0}\otimes \E_{n-1}^{0|1}\ \hookrightarrow \widetilde{\E}'_n\,.
\end{align*}
Let $\tilde{\tilde H}_{0,r}\in \E'_{1,0}$,  
$\tilde{H}_{i,r}\in \E_{n-1}^{1|0}$ ($0\le i\le n-2$) be the respective
commuting elements (we concentrate on the case $r\ge 1$). 
Let $\tilde{\tilde \theta}$ be the Miki automorphism of $ \E'_{1,0}$.  
We have two subalgebras of  $\E^0_n =U_q^h\widehat{\mathfrak{gl}}_{n}\subset \widetilde{\E}'_n$, 
\begin{align*}
&\cA=\pr\Bigl(\C[\tilde{\tilde \theta}^{-1}\bigl(\tilde{\tilde H}_{0,r}\bigr)
\mid r\ge 1]\otimes \C[\tilde{\theta}^{-1}\bigl(\tilde{H}_{i,r}\bigr)\mid r\ge 1, 0\le i\le n-2]\Bigr) \,,
\\
&\cB=\C[\theta^{-1}(H_{i,r})\ \mid r\ge 1, 0\le i\le n-1]\,.
\end{align*}
Both subalgebras are commutative. 
Note that both belong to the subalgebra generated by $E_{i,r}$, $0\le i\le n-1$, $r\in \Z$. 
From Proposition \ref{Hi1-n} and \ref{Hi1-2}, we know that 
\begin{align*}
\theta^{-1}(H_{i,1})\in \cA\cap \cB \quad (0\le i\le n-1)\,.
\end{align*}

\begin{prop}\label{A=B}
We have $\cA=\cB$. 
\end{prop}
\begin{proof}
We can reduce the problem from $U_q\widehat{\mathfrak{gl}}_n$
to $U_q\widehat{\mathfrak{sl}}_n$. 
So it suffices to prove the following statement:
Let $U_q\mathfrak{n}=\langle e_0,\ldots,e_{n-1}\rangle$ 
be the nilpotent subalgebra of $U_q\widehat{\mathfrak{sl}}_n$, 
and let $H_{i,r}\in  U_q\widehat{\mathfrak{sl}}_n$ ($r\ge 1, 1\le i \le n-1$)
be the Cartan-like commutative elements. Then the commutant of 
$H_{1,1}\,,\ldots, H_{n-1,1}$ in  $U_q\mathfrak{n}$ 
coincides with $\C[H_{i,r}\mid r\ge 1,1\le i\le n-1]$. 
The commutant obviously contains $\C[H_{i,r}\mid r\ge 1,1\le i\le n-1]$. 
Also for $q=1$ the statement is easy to see. Therefore it follows for generic $q$. 
\end{proof}

Proposition \ref{A=B} says that the projection  $\pr\bigl(\tilde{\theta}^{-1}(K_{i,\pm}(z))\bigr)$ 
belongs to the algebra generated by  $K_i^{\pm,=}(z)$. From \cite{FJMM1} we know that equalities \eqref{KKK1}, \eqref{KKK2} 
hold on the highest weight vectors of generic tensor products of Fock spaces. Therefore, by analytic continuation, 
these equalities hold on all highest weight vectors.
We note that the ideal of the projection $\pr$ is contained in 
the ideal which annihilates highest weight vectors. Therefore 
 Proposition \ref{simple proj} follows.

\subsection{The end of the proof}
Applying Proposition \ref{simple proj} repeatedly, we obtain \eqref{projection}.

\medskip

To show Theorem \ref{decompose}, it remains to show that the joint spectrum of $\ssK_{m}^{\pm,=}(z)$ is simple on $\oplus_{\bla} W(\bla)$. 
Suppose that the eigenvalues of $\ssK_{m}^{\pm,=}(z)$ 
coincide on some vectors $v\in W(\bla)$ and $w\in W(\bs \mu)$.  
Recall that $\bla^0$ is generic. 
Each Fock space entering as a factor 
in $W(\bla)$ depends on just one $\la^0_{i,j}$. 
It means that for each $i,j$ the eigenvalues should be 
the same on Fock spaces depending on $\la_{i,j}$ and $\mu_{i,j}$. 
The spectral parameters of these Fock spaces differ by $q_2^a$ with some 
$a\in\Z$. The eigenvectors in a Fock space are parametrized by partitions, 
and 
the corresponding eigenvalues are given by 
products over concave and convex boxes, 
see Lemma 3.4 in \cite{FJMM1}.
With the aid of this formula, 
it is easy to see that if the eigenvalues are the same on two 
partitions in two such Fock modules, then $a=0$ and the partitions are the same.

\bigskip

{\bf Acknowledgments.\ }
The research of BF is supported by 
the Russian Science Foundation grant project 16-11-10316. 
MJ is partially supported by 
JSPS KAKENHI Grant Number JP16K05183. 
EM is partially supported by a grant from the Simons Foundation  
\#353831.

We would like to thank Kyoto University for hospitality during our visit 
in summer 2017 when part of this work was completed.

\end{document}